\documentclass[a4paper, 12pt]{article}

\usepackage{amsmath, amsthm, amssymb}
\usepackage{graphicx}
\usepackage{enumitem}
\usepackage{authblk}
\usepackage[textsize=small, textwidth=2cm, color=yellow]{todonotes}
\usepackage{thmtools}
\usepackage{thm-restate}

\usepackage[utf8]{inputenc}
\usepackage[russian, english]{babel}
\usepackage[T1]{fontenc}
\usepackage[a4paper, margin=2.5cm]{geometry}
\usepackage{needspace}

\usepackage{libertine} 
\usepackage{inconsolata} 

\usepackage[hyphens]{url}
\usepackage[linktoc = all, hidelinks, colorlinks, unicode=true, pagebackref=true]{hyperref}
\usepackage[capitalise, compress, nameinlink, noabbrev]{cleveref}
\hypersetup{linkcolor={black}, citecolor={blue!70!black}, urlcolor={blue!70!black}}

\declaretheorem[name=Theorem, numberwithin=section]{theorem}
\declaretheorem[name=Lemma, sibling=theorem]{lemma}
\declaretheorem[name=Proposition, sibling=theorem]{proposition}

\declaretheorem[name=Remark, style=remark, sibling=theorem]{remark}

\def\cqedsymbol{\ifmmode$\lrcorner$\else{\unskip\nobreak\hfil
\penalty50\hskip1em\null\nobreak\hfil$\lrcorner$
\parfillskip=0pt\finalhyphendemerits=0\endgraf}\fi} 

\let\le\leqslant
\let\ge\geqslant
\let\leq\leqslant
\let\geq\geqslant

\makeatletter
\def\@setthanks{\vspace{-\baselineskip}\def\thanks##1{\@par\noindent##1\@addpunct.}\thankses}
\def\@setaddresses{%
  \par\nobreak\begingroup\footnotesize\interlinepenalty\@M\bigskip
  \def\author##1{}%
  \def\address##1##2{%
    \par\addvspace\medskipamount\noindent
    \@ifnotempty{##1}{\textit{##1}: }{\ignorespaces##2}%
  }%
  \def\email##1##2{\ifvmode\else\unskip; \fi\textit{email}: ##2}%
  \addresses\par
  \endgroup
}

\title{String graphs are quasi-isometric to planar graphs}

\author{James Davies\thanks{
The author was supported by the Alexander von Humboldt Foundation in the framework of the Alexander von Humboldt Professorship of Daniel Kráľ endowed by the Federal Ministry of Education and Research.}}

\affil{Leipzig University, Germany}

\date{}

\begin{document}

\maketitle

\begin{abstract}
    We prove that for every countable string graph $S$,
    there is a planar graph $G$ with $V(G)=V(S)$ such that 
    \[
    \frac{1}{23660800}d_S(u,v)
    \le
    d_G(u,v)
    \le
    162
    d_S(u,v)
    \]
    for all $u,v\in V(S)$, where $d_S(u,v)$, $d_G(u,v)$ denotes the distance between $u$ and $v$ in $S$ and $G$ respectively.
    In other words, string graphs are quasi-isometric to planar graphs.

    This theorem lifts a number of theorems from planar graphs to string graphs, we give some examples.
    String graphs have Assouad-Nagata (and asymptotic dimension) at most~2.
    Connected, locally finite, quasi-transitive string graphs are accessible.
    A finitely generated group $\Gamma$ is virtually a free product of free and surface groups if and only if $\Gamma$ is quasi-isometric to a string graph.
    
    Two further corollaries are that countable planar metric graphs and complete Riemannian planes are also quasi-isometric to planar graphs, which answers a question of Georgakopoulos and Papasoglu.
    For finite string graphs and planar metric graphs, our proofs yield polynomial time (for string graphs, this is in terms of the size of a representation given in the input) algorithms for generating such quasi-isometric planar graphs.

    We further extend our techniques to show that every complete Riemannian surfaces $\Sigma$ of bounded Euler genus has a triangulation $G\subset \Sigma$ such that $G^{(1)} \hookrightarrow \Sigma$ is a quasi-isometry, where $G^{(1)}$ is the simplicial 1-skeleton of $G$.
\end{abstract}

\section{Introduction}\label{sec:intro}

A \emph{string graph} is an intersection graph of a collection $\mathcal{S}$ of path-connected subsets of the plane, where the vertices are $\mathcal{S}$, and $S_1,S_2\in \mathcal{S}$ are adjacent if they intersect.
String graphs were first formally defined by Ehrlich, Even, and Tarjan \cite{ehrlich1976intersection} but were present in earlier works of Benzer~\cite{benzer1959topology} on the topology of genetic structures and of
Sinden \cite{sinden1966topology} on electrical networks realizable by printed circuit.
String graphs are the most general class of intersection graphs in the plane and significantly generalize planar graphs, which by the Koebe-Andreev-Thurston circle packing theorem \cite{Koebe36} can be viewed as intersection graphs of discs in the plane with disjoint interiors.
String graphs can be thought of as a dense or induced minor analogue of planar graphs.

There are several important structural results for string graphs.
For instance, Pach and T{\'o}th \cite{pach2002recognizing} proved that recognizing string graphs is decidable.
Fox and Pach \cite{fox2008separator,fox2010separator,fox2014applications} (see also \cite{matouvsek2014near}) proved a separator theorem for string graphs.
Every incomparability graph is a string graph and Fox and Pach \cite{fox2012string} proved a partial converse of this for dense string graphs.
This was used by Tomon \cite{tomon2023string} to prove that string graphs have the Erdős-Hajanl property.
Very recently, Abrishami, Briański, Davies, Du, Masaříková, Rzążewski, and Walczak \cite{ABRISHAMI2025string} classified string graphs with large chromatic number.

Every planar graph is a string graph, although the structure of string graphs is much more complicated.
For instance, while planar graphs can be recognized in linear time \cite{hopcroft1974efficient}, recognition of string graphs is NP-hard \cite{kratochvil1991string,schaefer2002recognizing}.
There are at most $c^nn!$ labelled planar graphs on $n$ vertices~\cite{denise1996random,norine2006proper}, while the number of labelled string graphs on $n$ vertices is $2^{(\frac{3}{4}+o(1)){n \choose 2}}$ \cite{pach2006many}.
Every planar graph is 4-colourable \cite{appel1977every,robertson1997four}, however there are triangle-free string graphs with arbitrarily large chromatic number \cite{pawlik2014triangle}.

We prove a partial converse to the fact that planar graphs are string graphs.
In a coarse sense, every string graph is equivalent to some planar graph.
For vertices $u,v$ of a graph $G$, we let $d_G(u,v)$ denote the distance between $u$ and $v$ in $G$.

\begin{theorem}\label{thm:stringmain}
    Let $S$ be a countable
    string graph.
    \footnote{We remark that \cref{thm:stringmain} does not hold for uncountable string graphs as can be seen by Wagner's \cite{wagner1967fastplattbare} characterization of uncountable planar graphs and the fact that the graph obtained from $|\mathbb{R}|$ 1-ended infinite paths by identifying their ends and then for each vertex adding $|\mathbb{R}|$ new adjacent vertices of degree 1 is a string graph.}
    Then there is a planar graph $G$ with $V(G)=V(S)$ such that 
    \[
    \frac{1}{23660800}d_S(u,v)
    \le
    d_G(u,v)
    \le
    162
    d_S(u,v)
    \]
    for all $u,v\in V(S)$.
\end{theorem}
In other words, string graphs are quasi-isometric (or even bi-Lipschitz equivalent) to planar graphs.
For finite string graphs the proof of \cref{thm:stringmain} yields a polynomial time (with respect to the size of a given representation of $S$ in the input) algorithm for generating such a planar graph $G$.
Unfortunately, there are string graphs that require exponentially sized representations~\cite{kratochvil1991stringexp}.
However, natural subclasses such as segment intersection graphs or more generally intersection graphs of convex compact path-connected subsets of $\mathbb{R}^2$, do have polynomially sized representations.
We remark that if the string graph $S$ is locally finite, then the resulting planar graph $G$ is also clearly locally finite.

\cref{thm:stringmain} has numerous applications.
We discuss some direct corollaries for string graphs.
Afterwards, we shall discuss further quasi-isometry results on metric planar graphs and complete Riemannian surfaces before giving a couple more applications.
\\

Asymptotic dimension was introduced by Gromov \cite{GroAsyInv} for metric spaces and groups (via their Cayley graphs).
This notion has had significant impact in geometric group theory, for instance Yu \cite{yu1998novikov} proved that finitely generated groups with finite asymptotic dimension satisfy the Novikov conjecture.
For a survey of asymptotic dimension from this point of view, see~\cite{bell2008asymptotic}.
In coarse graph theory, asymptotic dimension can be thought of as a coarse analogue of graph colouring.
There has recently been significant progress on the study of asymptotic dimension in the intersection of structural graph theory and geometric group theory~\cite{ABRISHAMI2025string,bernshteyn2025large,bonamy2023asymptotic,davies2025strong,distel2023proper,dvovrak2025asymptotic,liu2023assouad}. There are also applications of asymptotic dimension to graph colouring and $\chi$-boundedness, for instance hereditary classes of graphs with bounded asymptotic dimension are either $\chi$-bounded or contain all Burling graphs \cite{ABRISHAMI2025string}.

One of the most natural classes of graphs to consider is of course planar graphs.
Ostrovskii and Rosenthal \cite{ostrovskii2015metric} proved that planar graphs have bounded asymptotic dimension.
Fujiwara and Papasoglu \cite{fujiwara2021asymptotic} improved this bound to 3 and
J{\o}rgensen and Lang \cite{jorgensen2022geodesic} proved that planar graphs have asymptotic dimension and furthermore Assouad–Nagata dimension at most 2, which is tight due to the infinite planar grid.
Asymptotic dimension (and Assouad–Nagata dimension in the case of graphs) are persevered by quasi-isometries.
Thus, by \cref{thm:stringmain}, we obtain the following.

\begin{theorem}\label{col:asdim}
    String graphs have Assouad–Nagata dimension at most 2.
\end{theorem}

\cref{col:asdim} settles the 2-dimensional case of a conjecture of Davies, Georgakopoulos, Hatzel, and McCarty \cite{davies2025strong} that sphere intersection graphs in $\mathbb{R}^d$ have asymptotic dimension at most $d$.
For $d\ge 3$, such graphs have unbounded Assouad–Nagata dimension \cite{davies2025strong}.

Bonamy, Bousquet, Esperet, Groenland, Liu, Pirot, and Scott \cite{bonamy2023asymptotic} gave a proof that planar graphs have Assouad–Nagata dimension at most 2, which yields a polynomial time algorithm for computing the corresponding covers.
Thus, the corresponding covers in \cref{col:asdim} can also be computed in polynomial time (in terms of the size of the provided representation of the string graph given in the input).
For some algorithmic applications of such covers, see~\cite{awerbuch1990sparse,bonamy2025distributed,bonamy2025local,erschler2023assouad,filtser2024scattering,jia2005universal,lee2016separators,lokshtanov20241}.
\\

Dunwoody \cite{dunwoody2007planar} showed that every connected, locally finite, quasi-transitive, planar graph is accessible (see also \cite{esperet2024structure,hamann2018accessibility,hamann2018planar}).
MacManus \cite{macmanus2023accessibility} recently gave a remarkable extension of this theorem by showing that every connected, locally finite, quasi-transitive graph that is quasi-isometric to a planar graph is accessible.
Thus, by \cref{thm:stringmain}, Dunwoody's \cite{dunwoody2007planar} theorem extends from planar graphs to string graphs.

\begin{theorem}\label{thm:accessible}
    Every connected, locally finite, quasi-transitive string graphs is accessible.
\end{theorem}

There has been significant interest in planar Cayley graphs \cite{arzhantseva2004cayley,droms1998connectivity,droms2006infinite,dunwoody2007planar,georgakopoulos2017planar,georgakopoulos2019planar,georgakopoulos2019planar2,georgakopoulos2020planar,maschke1896representation,tucker1983finite} and Cayley graphs quasi-isometric to planar graphs \cite{bowditch2004planar,casson1994convergence,gabai1992convergence,macmanus2023accessibility,macmanus2024note,maillot2001quasi,mess1988seifert,tukia1988homeomorphic}.
Closely related to this, there are several results characterizing certain finitely presented groups $\Gamma$ in terms of being quasi-isometric to certain Riemannian surfaces.
For instance, $\Gamma$ is quasi-isometric to a complete
Riemannian plane if and only if $\Gamma$ is a virtual surface group \cite{bowditch2004planar,casson1994convergence,gabai1992convergence,maillot2001quasi,mess1988seifert,tukia1988homeomorphic}.
Maillot also showed that $\Gamma$ is quasi-isometric to a complete, simply
connected, planar Riemannian surface with non-empty geodesic boundary if and only if $\Gamma$ is virtually free.
In both these cases, $\Gamma$ is quasi-isometric to a planar graph since both surface groups and free groups have planar Cayley graphs.

MacManus \cite{macmanus2023accessibility} used his aforementioned theorem on accessibility to extend and unify these two theorems on finitely generated groups.
By \cref{thm:stringmain}, we obtain the following further extension of MacManus's \cite{macmanus2023accessibility} theorem
(our contribution is the equivalence between the first two bullets, the rest is given in \cite{macmanus2023accessibility}).

\begin{theorem}
    Let $\Gamma$ be a finitely generated group. Then the following are equivalent.
    \begin{itemize}
        \item $\Gamma$ is quasi-isometric to a string graph.
        \item $\Gamma$ is quasi-isometric to a planar graph.
        \item $\Gamma$ is quasi-isometric to a planar Cayley graph.
        \item Some finite index subgroup of $\Gamma$ admits a planar Cayley graph.
        \item $\Gamma$ is virtually a free product of free and surface groups.
    \end{itemize}
\end{theorem}

We now turn our attention to some further quasi-isometry results on metric planar graphs and complete Riemannian surfaces.\\

A \emph{metric graph} $H$ is a graph along with an assignment of non-negative edge lengths (in this paper, the length of edges in (non-metric) graphs is otherwise always uniformly equal to 1).
We can view metric graphs as a length space in which the edges are replaces by intervals of the corresponding lengths. When the edge lengths are within $D\subseteq (0,\infty)$, we say that $H$ is an $D$-metric graph.
The coarse geometry of metric graphs tends to be harder to work with than that of non-metric (simplicial) graphs \cite{liu2023assouad}.

Georgakopoulos and Papasoglu \cite{georgakopoulos2023graph} asked whether every planar metric graph is quasi-isometric to a planar graph.
As a straightforward application of \cref{thm:stringmain} (see \cref{Lem:metrictostring}), we answer this in positive (note that every planar metric graph can be turned into a {$(0,1]$-metric} planar graph by subdividing edges appropriately).
Such a planar graph $G$ can be generated in polynomial time.

\begin{theorem}\label{thm:metricmain}
    Let $H$ be a countable $(0,1]$-metric planar graph. Then there is a planar graph $G$ with $V(G)=V(H)$ such that 
    \[
    \frac{1}{47321600}d_H(u,v)
    \le
    d_G(u,v)
    \le
    162
    d_H(u,v)
    +1
    \]
    for all $u,v\in V(H)$.
\end{theorem}

Beyond the general interest in metric planar
graphs and also in groups acting on planar surfaces \cite{georgakopoulos2020planar},
one other motivation of Georgakopoulos and Papasoglu \cite{georgakopoulos2023graph} was their coarse Kuratowski conjecture that graphs forbidding $K_5$ and $K_{3,3}$ as $K$-fat minors are quasi-isometric to planar graphs.
Although \cref{thm:metricmain} and \cref{thm:stringmain} provide substantial progress towards this conjecture, we remark that the closely related graph minor analogue that graphs forbidding $H$ as a $K$-fat minor are quasi-isometric to $H$-minor-free graphs is false \cite{albrechtsen2025counterexample,davies2024fat}, (although it is known to be true for some small graphs $H$ \cite{albrechtsen2025fatK2n,albrechtsen2024characterisation,chepoi2012constant,fujiwara2023coarse,georgakopoulos2023graph,manning2005geometry,nguyen2025asymptotic}).
Such $K$-fat $H$-minor-free graphs need not even be quasi-isometric to $H'$-minor-free graphs for any graph $H'$ \cite{albrechtsen2025counterexample}.
\\

It is well known (see~\cite{bonamy2023asymptotic,burago2001course,creutz2022triangulating,georgakopoulos2026triangulating,ntalampekos2023polyhedral,saucan2008intrinsic}) that complete Riemannian surfaces are quasi-isometric to locally finite\footnote{A $(0,1]$-metric graph is \emph{locally finite} if every bounded radius ball contains at most a finite number of vertices.} $(0,1]$-metric graphs embeddable in the surface.
This was used for instance to show that complete Riemannian surfaces of bounded Euler genus (such as complete Riemannian planes) have Assouad–Nagata dimension at most 2 \cite{bonamy2023asymptotic} and in Maillot's~\cite{maillot2001quasi} aforementioned classification of finitely generated virtually free groups as being those quasi-isometric to certain planar Riemannian surfaces.
Ntalampekos and Romney \cite{ntalampekos2023polyhedral} also used a strong version for length surfaces as a step in their proof that any length surface is the Gromov–Hausdorff limit of polyhedral surfaces with controlled geometry. They gave further applications of this, including a new proof of the Bonk-Kleiner theorem \cite{bonk2002quasisymmetric} characterizing Ahlfors 2-regular quasispheres.

By \cref{thm:metricmain}, we obtain the following refinement, which significantly generalizes the previously discussed fact that finitely generated groups quasi-isometric to complete Riemannian planes are quasi-isometric to planar graphs.

\begin{theorem}\label{thm:Riemannain}
    Complete Riemannian planes are quasi-isometric to locally finite planar graphs.
\end{theorem}

In fact, we shall further extend our techniques to refine \cref{thm:Riemannain} in three ways.
Firstly, we obtain a quasi-isometry which is also an embedding of the graph into the Riemannian plane (as in~\cite{bonamy2023asymptotic,burago2001course,creutz2022triangulating,georgakopoulos2026triangulating,ntalampekos2023polyhedral,saucan2008intrinsic}).
Secondly, we also extend this result to Riemannian surfaces of bounded Euler genus.
Lastly, we show that such a graph and embedding can be chosen to be a triangulation.
A key ingredient in this last refinement is a recent result of Georgakopoulos and Vigolo \cite{georgakopoulos2026triangulating} that complete Riemannian surfaces admit triangulations by metric graphs that are quasi-isometric to the surface.
For an embedded graph $G\subset \Sigma$, we denote by $G^{(1)}$ its simplicial 1-skeleton (in which all edges have length 1).

\begin{theorem}\label{thm:RiemannianTriangulation}
    Let $\Sigma$ be a complete Riemannian surface (without boundary) of bounded Euler genus.
    Then $\Sigma$ has a triangulation $G\subset \Sigma$ such that $G^{(1)}\hookrightarrow \Sigma$ is a quasi-isometry.
\end{theorem}

The quasi-isometry in \cref{thm:RiemannianTriangulation} can be taken to be a $(10^6,10^6+10^4 g)$-quasi-isometry where $g$ is the Euler genus of the Riemannian surface.
The dependence on the Euler genus is necessary since for each fixed $M\ge 1$, there are complete Riemannian surfaces of genus at most $g$ that admit no such $(M, \Omega(\sqrt{\log g}))$-quasi-isometric simplicial triangulation \cite{davies2025riemannian}.
In particular, \cref{thm:RiemannianTriangulation} does not hold for complete Riemannian surfaces of unbounded genus~\cite{davies2025riemannian}.
Georgakopoulos and Vigolo \cite{georgakopoulos2026triangulating} recently showed that complete Riemannian surfaces with uniform nets also admit quasi-isometric simplicial triangulations as in \cref{thm:RiemannianTriangulation}, and
Bowditch~\cite{bowditch2020bilipschitz} showed that complete Riemannian manifolds of bounded geometry also admit such simplicial triangulations.

We remark that while we prove \cref{thm:RiemannianTriangulation} for Riemannian surfaces without boundary, it is possible to extend \cref{thm:RiemannianTriangulation} to Riemannian surfaces with boundary with more technical arguments and with the same resulting $(10^6,10^6+10^4 g)$-quasi-isometry bounds.
\\

We expect our quasi-isometry results to have many further applications. We quickly give two more concerning $\ell_1$-embeddings and the coarse Menger conjecture.
\\

A well-known conjecture (see \cite{gupta2004cuts}) states that planar graphs admit embeddings into $\ell_1$ with bounded distortion.
\cref{thm:stringmain} implies that this is equivalent to string graphs having embeddings into $\ell_1$ with bounded distortion.
Rao \cite{rao1999small} showed that $n$-vertex planar graphs have embeddings into $\ell_1$ with distortion at most $O(\sqrt{\log n})$, and \cref{thm:stringmain} extends this to string graphs.

\begin{theorem}
    Every $n$-vertex string graph has a $O(\sqrt{\log n})$-distortion embedding into $\ell_1$.
\end{theorem}
For a more extensive survey and discussion of the applications of such embeddings, see~\cite{gupta2004cuts}.
\\

Albrechtsen, Huynh, Jacobs, Knappe and Wollan~\cite{albrechtsen2024menger}, and independently Georgakopoulos
and Papasoglu \cite{georgakopoulos2023graph} conjectured a coarse analogue of Menger's \cite{menger1927allgemeinen} theorem that for a graph $G$ containing two vertex sets $S,T$, either there are $k+1$ paths between $S$ and $T$ pairwise at distance at least $c$, or there is a set of at most $k$ balls in $G$ of radius at most $f(c,k)$ such that every path between $S$ and $T$ intersects one of these balls.
This turns out to be false in general~\cite{nguyen2025counterexample}, and even in the weak version where up to $f(c,k)$ balls are allowed \cite{nguyen2025asymptotic4}.
However, it is true for certain classes such as those of bounded path-width \cite{nguyen2025asymptotic5}.
A particularly interesting open case is that of planar graphs.
Nguyen, Scott, and Seymour \cite{nguyen2025asymptotic6} proved a version for paths between vertices on the outerface of a planar graph. By \cref{thm:metricmain}, we obtain the following extension of their result to metric planar graphs\footnote{We remark that there is a technicality of preserving vertices on the outerface, however, the proof of \cref{thm:metricmain} gives this.} (they proved the same theorem for (simplicial) graphs, just with a much better constant than $10^{10}$).

\begin{theorem}\label{Menger}
    Let $c, k \ge 0$, and let $G$ be a metric planar graph, with vertex sets $S, T$ on the outerface. Then either:
    \begin{itemize}
        \item there are $k + 1$ paths between $S$ and $T$, pairwise at distance least $c$; or
        \item there is a set of at most $k$ balls in $G$ of radius at most $10^{10}k^3c$ of G, such that every path between $S$ and $T$ intersects one of these balls.
    \end{itemize}
\end{theorem}

One can also obtain versions of \cref{Menger} for string graphs and complete Riemannian surfaces homeomorphic to a disk.
\\

In \cref{Sec:Pre} we introduce basic definitions, deduce \cref{thm:metricmain} from \cref{thm:stringmain}, and also reduce proving \cref{thm:stringmain} to its finite case (\cref{thm:stringmainfinite}).
We also introduce the necessary framework for our approach to proving \cref{thm:stringmain}, namely through viewing string graphs as region intersection graphs, and the notion of ``impressions'' which provides a convenient way of finding quasi-isometries from region intersection graphs.
The rest of the proof in \cref{sec:outer,sec:encase,sec:string} is essentially a delicate technical induction.
In \cref{sec:outer} we prove a weakening of \cref{thm:stringmain} for the subclass outerstring graphs (\cref{thm:outerstring}).
This is used in \cref{sec:encase} to find a technical structure tailored for use in the inductive argument that we call an ``encasing'' (see \cref{lem:encase4}).
Finally, in \cref{sec:string}, we prove \cref{thm:stringmain} using \cref{lem:encase4}.
Lastly, in \cref{sec:Riemannian}, we build on one version of our main technical result (\cref{thm:stringimpression}) to prove \cref{thm:RiemannianTriangulation}.

\begin{remark}
More recently, Chang, Conroy, Tan, and Zheng \cite{chang20251} have independently obtained another proof of a result very similar to \cref{thm:stringmain}.
While our proof is self contained and builds the quasi-isometry from first principles, theirs adapts a shortcut partition construction using gridtrees developed in \cite{chang2023covering}.
We remark that their motivation and applications differs significantly from the present paper and we encourage the interested reader to also see \cite{chang20251} for further applications.
For instance, they construct compact (approximate) distance oracles for string graphs and also prove that for any $\epsilon >0$, $n$-vertex string graphs have $(1+\epsilon , O(\epsilon^{-4} \log^{18}n))$-quasi-isometric metric planar graphs.

In \cite{davies2025coarse}, we will build on the work of the present paper and \cite{chang20251} to prove versions of \cref{thm:metricmain} and \cref{thm:stringmain} for minor-free graphs.

\end{remark}

\section{Preliminaries}\label{Sec:Pre}

For~$M, A \ge 0$ with $M \geq 1$, an \emph{$(M,A)$-quasi-isometry} from a metric space~$G$ to another metric space~$H$ is a map~$f: G \to H$ such that
\begin{enumerate}
    \item $M^{-1} d_G(u, v) - A \leq d_H(f(u),f(v)) \leq M  d_G(u,v) + A$ for every~$u, v \in G$, and
    \item for every $h\in H$, there exists some $v\in G$ with $d_H(h,f(v)) \leq A$.
\end{enumerate}
We say that $G$ is \emph{$(M,A)$-quasi-isometric} to $H$ if there exists a $(M,A)$-quasi-isometry from $G$ to $H$.
Two metric spaces $G$ and $H$ are \emph{quasi-isometric} if they are $(M,A)$-quasi-isometry for some $M,A\ge 0$.
Quasi-isometries for graphs are of course defined in the same expected way.

For a set of vertices $A$ of a graph $G$, we let $G[A]$ denote the \emph{induced subgraph} of $G$ on vertex set $A$.
We say that $A$ is \emph{connected} if the induced subgraph $G[A]$ is connected.
We let $N_G(A)$ denote the \emph{neighbourhood} of the vertex set $A$ in $G$ (the vertices of $V(G)\backslash A$ adjacent to a vertex of $A$).
For a positive real $t$, we let $N^t_G[A]$ be the vertices at distance at most $t$ from $A$.
If $A=\{u\}$, then we simply use $N_G(u)$, $N^t_G[u]$.
The \emph{weak diameter} of $A$ in $G$ is equal to the maximum distance (in $G$) between two vertices of $A$.
Given a set $F$ and another collection of sets $\mathcal{H}$, we let $I_{\mathcal{H}}(F)$ denote the subcollection of elements $H\in \mathcal{H}$ that intersect $F$.

Given a graph $G$ and a collection of connected vertex sets $\mathcal{H}$ of $G$, we let $\textbf{RIG}(G,\mathcal{H})$ be the graph with vertex set $\mathcal{H}$ where, $H_1,H_2\in \mathcal{H}$ are adjacent if they intersect.
We call $\textbf{RIG}(G,\mathcal{H})$ a \emph{region intersection graph}.
Region intersection graphs were introduced by Lee \cite{lee2016separators} as a generalization of string graphs (usually we consider region intersection graphs $\textbf{RIG}(G,\mathcal{H})$ in which $G$ forbids some graph as a minor, since otherwise the class contains all graphs).
Finite string graphs are exactly finite region intersection graphs $\textbf{RIG}(G,\mathcal{H})$ for which $G$ is planar~\cite{lee2016separators} and this is how we will view string graphs to prove \cref{thm:stringmain} (or its finite version, \cref{thm:stringmainfinite}).

If $\mathcal{M}$ is a collection of disjoint connected vertex sets in $G$, then we let $\textbf{IM}(G,\mathcal{M})$ be the graph with vertex set $\mathcal{M}$ where $M_1,M_2\in \mathcal{M}$ are adjacent if there is an edge in $G$ between $M_1$ and $M_2$ (which case, we say that they \emph{touch}).
This is one way to define \emph{induced minors} of a graph $G$.

Let $G$ be a graph and $\mathcal{H}$ a collection of connected vertex sets of $G$.
If every edge of $G$ is contained in $G[H]$ for some $H\in \mathcal{H}$, then we say that $\mathcal{H}$ \emph{spans} $G$.
More generally, we say that $\mathcal{H}$ \emph{$z$-spans} $G$ if for every edge $uv$ of $G$, there exists $H_u,H_v \in \mathcal{H}$ with $u\in H_u$ and $v\in H_v$, such that $d_{\textbf{RIG}(G,\mathcal{H})}(H_u,H_v) \le z$.
Note that $\mathcal{H}$ spans $G$ if and only if it 0-spans $G$.

For a graph $G$ and two collections of connected vertex sets $\mathcal{H}, \mathcal{M}$ of $G$ with the sets in $\mathcal{M}$ being disjoint, we say that $(G,\mathcal{M})$ is a \emph{$(x,y)$-impression} of $(G,\mathcal{H})$ if
\begin{itemize}
    \item for every $M\in \mathcal{M}$, $I_\mathcal{H}(M)$ has weak diameter at most $x$ in $\textbf{RIG}(G,\mathcal{H})$,
    \item for every $H\in \mathcal{H}$, $I_\mathcal{M}(H)$ has weak diameter at most $y$ in $\textbf{IM}(G,\mathcal{M})$, and
    \item $\bigcup_{M\in \mathcal{M}} M = V(G)$.
\end{itemize}

Impressions provide a convenient way of obtaining a quasi-isometry from a region intersection graph $\textbf{RIG}(G,\mathcal{H})$ to a minor of $G$.

\begin{proposition}\label{prop:impression}
    Let $(G,\mathcal{M})$ be a $(x,y)$-impression of $(G,\mathcal{H})$ where $\mathcal{H}$ $z$-spans $G$.
    Then there is a function $f:\mathcal{H} \to \mathcal{M}$ such that for all $H_1,H_2\in \mathcal{H}$, we have 
    \[
    \frac{1}{x+z} d_{\textbf{RIG}(G,\mathcal{H})}(H_1,H_2)
    -
    \frac{x}{x+z}
    \le
    d_{\textbf{IM}(G,\mathcal{M})}(f(H_1),f(H_2)) 
    \le 
    2y d_{\textbf{RIG}(G,\mathcal{H})}(H_1,H_2),
    \]
    and for every $M\in \mathcal{M}$ there exists some $H\in \mathcal{H}$ such that $d_{\textbf{IM}(G,\mathcal{M})}(f(H),M) \le y$.
    
    Furthermore, $f$ can be chosen to be any function $f:\mathcal{H} \to \mathcal{M}$ such that $f(H)\in I_{\mathcal{M}}(H)$ for all $H\in \mathcal{H}$.
\end{proposition}

\begin{proof}
    Let $f:\mathcal{H} \to \mathcal{M}$ be chosen so that $f(H)\in I_{\mathcal{M}}(H)$ for all $H\in \mathcal{H}$.
    Clearly such a function exists since $I_{\mathcal{M}}(H)$ is always non-empty.
    
    If $H_1$ and $H_2$ are adjacent in $\textbf{RIG}(G,\mathcal{H})$, then they intersect. So then $I_\mathcal{M}(H_1\cup H_2)$ is a connected set in $\textbf{IM}(G,\mathcal{M})$ and furthermore has weak diameter at most $2y$.
    It follows for (possibly non-adjacent $H_1,H_2\in \mathcal{H}$) that $d_{\textbf{IM}(G,\mathcal{M})}(f(H_1),f(H_2)) \le 2y d_{\textbf{RIG}(G,\mathcal{H})}(H_1,H_2)$.

    Suppose that $d_{\textbf{IM}(G,\mathcal{M})}(f(H_0),f(H_t)) =t$ and let $M_0,\ldots ,M_t$ be a path in $\textbf{IM}(G,\mathcal{M})$ between $M_0=f(H_0)$ and $M_t=f(H_t)$.
    For each $1\le i < t$, there exists some $u_iv_i\in E(G)$ with $u_i \in M_{i-1}$ and $v_i\in M_i$.
    For each $1\le i < t$, let $H_i',H_i\in \mathcal{H}$ be such that $u_i\in V(H_i')$, $v_i\in V(H_i)$, and the distance between $H_i'$ and $H_i$ in $\textbf{RIG}(G,\mathcal{H})$ is at most $z$.
    For each $1\le i \le t$, we have that $H_{i-1},H_i'\in I_{\mathcal{H}}(M_{i-1})$.
    Thus, there is a path in $\textbf{RIG}(G,\mathcal{H})$ between $H_0$ and $H_t$ of length at most $(t+1)x +tz$.
    It now follows that for all $H_1,H_2\in \mathcal{H}$, we have 
    \[
    \frac{1}{x+z} d_{\textbf{RIG}(G,\mathcal{H})}(H_1,H_2)
    -
    \frac{x}{x+z}
    \le
    d_{\textbf{IM}(G,\mathcal{M})}(f(H_1),f(H_2)) 
    \le 
    2y d_{\textbf{RIG}(G,\mathcal{H})}(H_1,H_2).
    \]

    Lastly, consider some $M\in \mathcal{M}$. Choose some $H\in I_{\mathcal{H}}(M)$.
    Then $M,f(H)\in I_{\mathcal{M}}(H)$, and thus $M$ is at distance at most $y$ from $f(M)$ in $\textbf{IM}(G,\mathcal{M})$.
\end{proof}

In most but not all cases we consider, we will actually have that $\mathcal{H}$ spans $G$, which means that we simply have $z=0$ in \cref{prop:impression}.
Given $(G,\mathcal{H})$, we could simply reduce to the case the $\mathcal{H}$ spans $G$ by removing edges of $G$ not contained in $G[H]$ for any $H\in \mathcal{H}$ (which is something we often do).
However, adding edges to $G$ can also make finding impressions easier.
In fact, for $G$ planar, it is convenient for inductive arguments that we can also add new connected sets to $\mathcal{H}$ in a controlled way so that our final impression can be obtained from a related impression.

\begin{lemma}\label{lem:impression2}
    Let $(G,\mathcal{M})$ be a $(x,y)$-impression of $(G,\mathcal{H}')$.
    If $\mathcal{H}\subseteq \mathcal{H}'$ has the property that for every $H\in \mathcal{H'} \backslash \mathcal{H}$ we have that $I_{\mathcal{H}}(H)$ has weak diameter at most $k$ in $\textbf{RIG}(G,\mathcal{H})$ and $V(G)=\bigcup_{H\in \mathcal{H}} V(H)$, then $(G,\mathcal{M})$ is a $(kx,y)$-impression of $\textbf{RIG}(G,\mathcal{H})$.
\end{lemma}

\begin{proof}
    For each $H\in \mathcal{H}\subseteq \mathcal{H}'$, we have that $I_\mathcal{M}(H)$ has weak diameter at most $y$ in $\textbf{IM}(G,\mathcal{M})$ since $(G,\mathcal{M})$ is a $(x,y)$-impression of $\textbf{RIG}(G,\mathcal{H}')$.
    For $M\in \mathcal{M}$, we have that $I_{\mathcal{H}'}(M)$ has weak diameter at most $x$ in $\textbf{RIG}(G,\mathcal{H}')$.
    Since $V(G)=\bigcup_{H\in \mathcal{H}} V(H)$ and $I_{\mathcal{H}}(H)$ has weak diameter at most $k$ in $\textbf{RIG}(G,\mathcal{H})$ for each $H\in \mathcal{H'} \backslash \mathcal{H}$, it further follows that $I_{\mathcal{H}}(M)$ has weak diameter at most $kx$ in $\textbf{RIG}(G,\mathcal{H})$.
    Hence $(G,\mathcal{M})$ is a $(kx,y)$-impression of $\textbf{RIG}(G,\mathcal{H})$, as desired.
\end{proof}

It is straight forward to make the quasi-isometry $f$ in \cref{prop:impression} into a bijective map by taking a minor and possibly adding some pendent edges.
Note in particular that for countable graphs this preserves planarity.

\begin{lemma}\label{lem:quasibi}
    Let $G,H$ be graphs and let $f:V(G)\to V(H)$ be a function such that for all $u,v\in V(G)$, we have
    \[
    \frac{1}{x_1}d_G(u,v) - \frac{x_2}{x_1}
    \le 
    d_H(f(u), f(v))
    \le 
    x_3d_G(u,v),
    \]
    and for every $v\in V(H)$ there exists some $u\in V(G)$ such that $d_H(f(u),v) \le x_4$.

    Then there exists some graph $H'$ obtained from $H$ by adding vertices adjacent to individual vertices of $V(H)$ and then taking a minor, and some bijection $f':V(G)\to V(H')$ such that for all $u,v\in V(G)$, we have
    \[
    \frac{1}{2x_4(x_1 + x_2)}
    d_G(u,v)
    \le 
    d_{H'}(f'(u), f'(v))
    \le 
    (x_3 +2)
    d_G(u,v).
    \]
\end{lemma}

\begin{proof}
    Let $A=f(V(G))$.
    Since for every $v\in V(H)$ there exists some $u\in V(G)$ such that $d_H(f(u),v) \le x_4$, we can choose a forest $F$ contained in $H$ as a subgraph with $V(F)=V(H)$ such that every connected component of $F$ contains exactly one vertex of $A$ and every vertex of $F$ is at distance at most $x_4$ from $A$ in $F$.
    Let $H^*$ be obtained from $H$ by for each $a\in A$, adding $|f^{-1}(a)|-1$ vertices to $H$ that are adjacent to $a$ only.
    Choose $f':V(G)\to V(H^*)$ so that for every $a\in A$, $f^{-1}(a)$ is mapped bijectivly to $a$ and the newly added $|f^{-1}(a)|-1$ vertices adjacent to $a$.
    Let $H'$ be obtained from $H^*$ by contracting the edges of $F$ onto $A$.
    Then $f'$ is a bijection between $G$ and $H'$.

    For every $u,v\in V(G)$, we have
    \[
    \frac{1}{2x_4}
    d_{H}(f(u),f(v))
    \le 
    d_{H'}(f'(u), f'(v))
    \le 
    d_{H}(f(u),f(v)) +2.
    \]
    For distinct $u,v\in V(G)$, we clearly have $d_{H}(f(u),f(v)) +2 \le 3d_{H}(f(u),f(v))
    \le x_3d_G(u,v) +2
    \le (x_3+2)d_G(u,v) 
    $.
    For $t\ge 0$, we have that 
    $
    \frac{1}{x_1+x_2}t
    \le 
    \max\{
    \frac{1}{x_1}t-\frac{x_2}{x_1},
    1\}$.
    Since $f'$ is bijective, it therefore follows that $\frac{1}{2x_4(x_1 + x_2)} d_G(u,v) \le d_{H'}(f'(u), f'(v))$.
\end{proof}

\cref{thm:stringmain} for finite string graphs will be quickly deduced in \cref{sec:string} from our main technical result (\cref{thm:stringinduct}) by applying \cref{lem:impression2}, \cref{prop:impression}, and \cref{lem:quasibi} in order.

Next we show that \cref{thm:stringmain} implies \cref{thm:metricmain}.
This is immediate from the following lemma. 

\begin{lemma}\label{Lem:metrictostring}
    Let $H$ be a $(0,1]$-metric planar graph.
    Then there is a string graph $S$ with $V(S)=V(H)$ such that
    \[
    \frac{1}{2}d_H(u,v)
    \le
    d_S(u,v)
    \le
    d_H(u,v) + 1
    \]
    for all $u,v\in V(H)$.
\end{lemma}

\begin{proof}
    Consider $H$ embedded in the plane.
    For each $u\in V(H)$, let $S_u$ be the subset of the plane consisting of all points of $H$ at distance at most 1 in $H$ from $u$.
    Let $S$ be the string graph with vertex set $V(H)$ in which each $u\in V(S)=V(H)$ is represented by $S_u$.
    Consider $u,v\in V(H)$
    and let $P$ be a geodesic path from $u$ to $v$ in $H$.
    Choose vertices $x_0,\ldots ,x_t$ of $H$ along $P$ so that $x_0=u$, $x_t=v$ and for each $0\le i <t$, the distance along $P$ between $x_i$ and $x_{i+1}$ is in $[1,2]$ (or if $d_H(u,v) < 1$, we just take $x_0=u$, $x_1=v$).
    So, $\lceil \frac{1}{2}d_H(u,v) \rceil \le  t\le \lfloor d_H(u,v) \rfloor$.
    The lemma now follows since $x_i$ is adjacent to $x_{i+1}$ in $S$ for each $0\le i <t$.
\end{proof}

We remark that with a very similar proof to \cref{Lem:metrictostring}, it is easy to extend \cref{thm:stringmain} to $(0,1]$-metric string graphs (again with slightly worse bounds).

For our inductive argument, it is more convenient to prove \cref{thm:stringmain} for finite string graphs.
\cref{sec:outer}, \cref{sec:encase}, and \cref{sec:string} are dedicated to proving \cref{thm:stringmain} for finite string graphs (\cref{thm:stringmainfinite} below), which in particular, will be quickly deduced in \cref{sec:string} from \cref{thm:stringinduct} (which is a more technical result tailored for an inductive argument) by applying \cref{lem:impression2}, \cref{prop:impression}, and \cref{lem:quasibi} in order.

\begin{theorem}\label{thm:stringmainfinite}
    Let $S$ be a finite string graph. Then there is a planar graph $G$ with $V(G)=V(S)$ such that 
    \[
    \frac{1}{23660800}d_S(u,v)
    \le
    d_G(u,v)
    \le
    162
    d_S(u,v)
    \]
    for all $u,v\in V(S)$.
\end{theorem}

We finish this section by showing that \cref{thm:stringmainfinite} implies \cref{thm:stringmain}.

\begin{proof}[Proof of \cref{thm:stringmain} assuming \cref{thm:stringmainfinite}.]
    Clearly it is enough to consider the case that $S$ is connected.
    Let $V(S)=\{x_k:k\ge 1\}$ and for each $i\le 1$, let $X_i = \{x_k: 1\le k \le i\}$.
    For each $i\ge 1$, let $\mathcal{G}_i$ be the collection of embedded planar graphs $G'$ that are topologically inequivelant after possibly relabeling vertices of $V(G')\backslash X_i$,  such that $X_i\subseteq V(G')\subseteq V(S)$, $|V(G')| \le i + \sum_{x,y\in X_i} 162d_S(x,y)$, and for every $x,y\in X_i$, we have that 
    \[
    \frac{1}{23660800}d_S(x,y)
    \le
    d_{G'}(x,y)
    \le
    162
    d_S(x,y).
    \]
    Note that each $\mathcal{G}_i$ is non-empty by \cref{thm:stringmainfinite} (applied to some $S[Y]$ for some $Y$ with $X_i\subseteq Y$, $|Y|\le |X_i| + \sum_{x,y\in X_i} d_S(x,y)$ such that $d_{S[Y]}(x,y)=d_S(x,y)$ for all $x,y\in X_i$).
    Also, each $|\mathcal{G}_i|$ is finite.
    Let $T$ be the graph with vertex set $\bigcup_{i=1}^\infty \mathcal{G}_i$ where each $G''\in \mathcal{G}_{i+1}$ is adjacent to some $G'\in \mathcal{G}_i$ such that $G'$ is topologically equivalent to a restriction of $G''$ after possibly relabeling vertices of $V(G'')\backslash X_i$.
    Observe that by considering shortest paths between vertices of $X_i$, such a $G'$ always exists.

    By Kőnig's infinity lemma \cite{konig1927schlussweise}, $T$ has an infinite path $G_1,G_2,\ldots$ in which $G_i\in \mathcal{G}_i$ for each $i\ge 1$.
    Note that for each $i\ge 1$, there exists some $R$ so that $V(G_i)\subseteq V(G_r)$ for all $r\ge R$.
    We obtain the desired countable planar graph $G$ by iteratively extending an embedded planar graph topically equivalent to $G_i$ to one topologically equivalent to $G_{i+1}$ for each $i\ge 1$ in order.
\end{proof}

\section{Outerstring graphs}\label{sec:outer}

Outerstring graphs are equivalent to region intersection graphs $\textbf{RIG}(G,\mathcal{H})$ for which $G$ is planar and every $H\in \mathcal{H}$ contains a vertex of the outer face of $G$.
They can be thought of as an ``outerplanar'' analogue of string graphs, and are thus much simpler.
The main aim of this section is to prove the following theorem on finding certain restricted impressions of such region intersection graphs.
This will be crucial to our arguments on string graphs (or strictly speaking, \cref{lem:outerstring} and \cref{lem:outerplanar} which easilly imply the following theorem will be).

\begin{theorem}\label{thm:outerstring}
    Let $G$ be a finite planar graph and $\mathcal{H}$ be a spanning collection of connected vertex sets in $G$, each of which contains a vertex incident to the outer face of $G$.
    Then there is a $(770,9)$-impression $(G, \mathcal{M})$ of $(G,\mathcal{H})$ such that each set in $\mathcal{M}$ contains a vertex incident to the outer face of $G$.
\end{theorem}

Note that if $(G, \mathcal{M})$ is an impression of $(G, \mathcal{H})$ where each graph in $\mathcal{M}$ contains a vertex incident to the outer face of $G$, then $\textbf{IM}(G,\mathcal{M})$ is an outerplanar graph.

\begin{figure}
    \centering
    \includegraphics[width=0.3\linewidth]{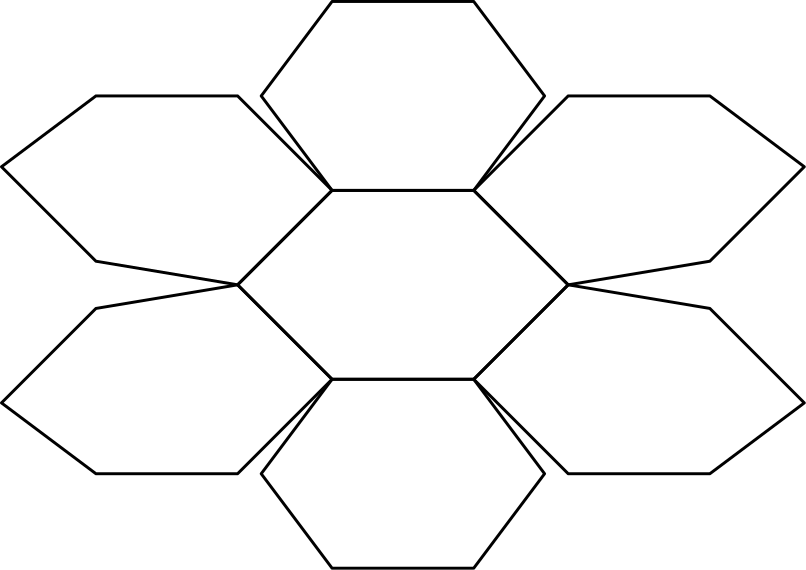}
    \caption{The graph $F_6$.}
    \label{fig:nocacti}
\end{figure}

By \cref{prop:impression}, \cref{thm:outerstring} implies that outerstring graphs are quasi-isometric to outerplanar graphs. This result is already easily implied by a theorem that graphs forbidding $K_{2,3}$ as a fat minor are quasi-isometric to cacti \cite{chepoi2012constant,fujiwara2023coarse}, where a \emph{cacti} is a planar graph in which every edge is on the outer face, or equivalently, a graph in which each maximal 2-connected subgraph is either an edge of a cycle.
One might hope that \cref{thm:outerstring} could be improved so that $\textbf{IM}(G,\mathcal{M})$ is also a cacti.
However this is not possible. For every $\ell \ge 3$, let $F_\ell$ be the outerplanar graph consisting of a cycle of length $\ell$ where for every edge of this cycle, we identify the edge with an edge of another cycle of length $\ell$ (see \cref{fig:nocacti} for an illustration of the graph $F_6$).
Then it can easily be seen that there is no $x,y$ such that every $(F_\ell, \{ \{u,v\} :uv \in E(F_\ell)\})$ has a $(x,y)$-impression $(F_\ell, \mathcal{M}_\ell )$ in which $\textbf{IM}(F_\ell,\mathcal{M}_\ell)$ is a cacti.
It is however a nice exercise to show that outerplanar graphs are quasi-isometric to cacti. With this, \cref{thm:outerstring} and \cref{prop:impression} then yeild another proof that outerstring graphs are quasi-isometric to cacti.

To prove \cref{thm:outerstring}, we begin with the easier case of when $G$ is simply an outerplanar graph (so when every vertex of $G$ is incident to the outer face of $G$).
For this, we first define separations.
For a connected graph $G$ and connected sets $X,Y$ of $G$, we say that $(X,Y)$ is a \emph{separation} of $G$ if $X\cup Y=V(G)$, both $X\backslash Y$ and $Y\backslash X$ are non-empty, and $G \backslash (X\cap Y)$ is disconnected.
If $|X\cap Y|\le k$, then we say that $(X,Y)$ is a \emph{$k$-separation} of $G$.

\begin{lemma}\label{lem:outerplanar}
    Let $G$ be a finite outerplanar graph and $\mathcal{H}$ a spanning collection of connected vertex sets in $G$.
    Then there is a $(11,9)$-impression $(G, \mathcal{M})$ of $(G,\mathcal{H})$.
\end{lemma}

\begin{proof}
    Clearly it is enough to handle the case that $G$ is connected.
    Choose some $H_0\in \mathcal{H}$ and for every non-negative integer $i$, let $L_i$ be the vertices of $\textbf{RIG}(G,\mathcal{H})$ at distance $i$ from $H_0$ in $\textbf{RIG}(G,\mathcal{H})$.
    For each $i$ we define the following two vertex sets in $G$.
    Let $D_i=\bigcup_{H\in L_i}H$, and let $A_i=D_i\backslash D_{i-1}$ (where $D_{-1}=\emptyset$).
    For each $A_i$, let $\mathcal{M}_i$ be the maximal connected vertex sets of $G[A_i]$.
    Let $\mathcal{M}=\bigcup_{i=0}^\infty \mathcal{M}_i$.
    See \cref{fig:outerplanar} for an illustration.
    We must show that $(G, \mathcal{M})$ is a $(11,9)$-impression of $(G,\mathcal{H})$.

\begin{figure}
    \centering
    \includegraphics[width=0.65\linewidth]{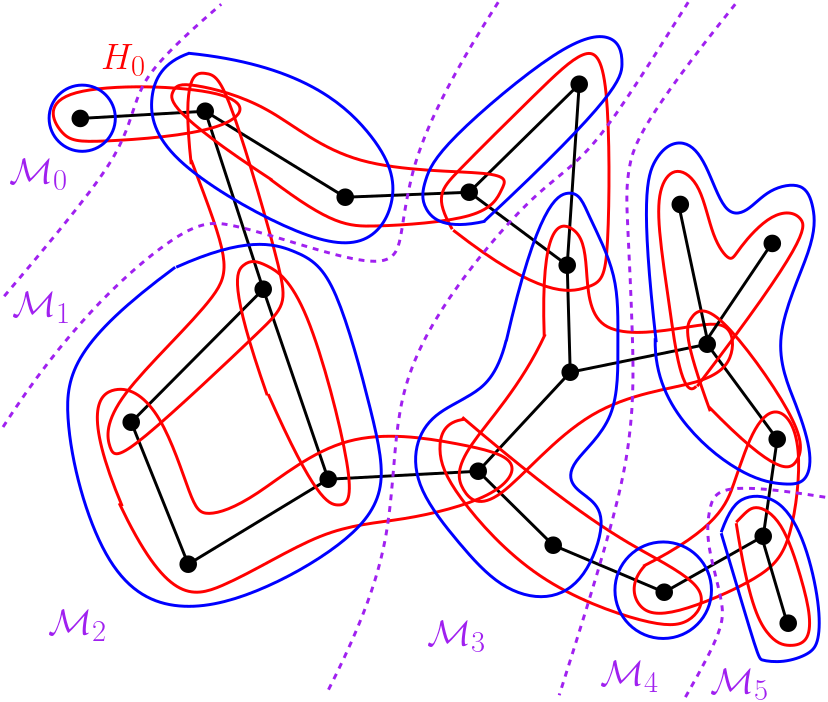}
    \caption{An illustration of the resulting collection of disjoint connected sets $\mathcal{M}$.
    The connected sets in $\mathcal{H}$ are the red sets and the sets in $\mathcal{M}$ are the blue ones. The dashed purple lines indicate the partition of $\mathcal{M}$ into $\bigcup_{i=0}^\infty \mathcal{M}_i$.}
    \label{fig:outerplanar}
\end{figure}

    Next, we show for every $M\in \mathcal{M}$ that $I_\mathcal{H}(M)$ has weak diameter at most $11$ in $\textbf{RIG}(G,\mathcal{H})$.
    Let $M\in \mathcal{M}$. Then, for some non-negative integer $i$, we have that $M\in \mathcal{M}_i$.
    We may assume that $i\ge 2$, since otherwise we clearly have that $I_\mathcal{H}(M)$ has weak diameter at most $4$ in $\textbf{RIG}(G,\mathcal{H})$.
    Since $G$ is outerplanar, observe that there exists a $2$-seperation $(X,Y)$ of $G$ with $A_{i-2}\subseteq X$, $X\cap Y\subseteq A_{i-1}$, and $M \subseteq Y$.
    To see this one can apply Menger's theorem between $\bigcup_{0\le j \le i-2} A_j$ and the connected component of $\bigcup_{j\ge i} A_j$ that contains $M$.
    For each $x\in X\cap Y$, let $\mathcal{H}_x = \{H\in L_{i-1} : x\in V(H)\}$.
    So, $\mathcal{H}_x$ is a clique in $\textbf{RIG}(G,\mathcal{H})$.
    Furthermore, each $H\in I_\mathcal{H}(M)$ is at distance at most $2$ from $\bigcup_{x\in X \cap Y} \mathcal{H}_x$ in $\textbf{RIG}(G,\mathcal{H})$.
    Since $|X\cap Y|\le 2$ and $\textbf{RIG}(G,\mathcal{H})[I_\mathcal{H}(M)]$ is connected, it now follows that $I_\mathcal{H}(M)$ has weak diameter at most $(2+1+2)+1+(2+1+2)=11$ in $\textbf{RIG}(G,\mathcal{H})$.

    Now, we show for every $H\in \mathcal{H}$ that $I_\mathcal{M}(H)$ has weak diameter at most 9 in $\textbf{IM}(G,\mathcal{M})$. This is similar to the previous argument.
    Let $H\in \mathcal{H}$.
    Then for some $i\ge 0$, we have that $H\in L_i$.
    We may assume that $i\ge 3$, since otherwise we clearly have that $I_\mathcal{M}(H)$ has weak diameter at most $6$ in $\textbf{IM}(G,\mathcal{M})$.
    Similarly to before, by applying Menger's theorem between
    $\bigcup_{0\le j \le i-3} A_j$ and the connected component of $\bigcup_{j\ge i-1} A_j$ that contains $H$,
    there exists a $2$-separation $(X,Y)$ of $G$ with $A_{i-3}\subseteq X$, $X\cap Y\subseteq A_{i-2}$, and $H \subseteq Y$.
    For each $x\in X\cap Y$, choose some $M_x\in \mathcal{M}_{i-2}$ with $x\in M_x$.
    Then in $\textbf{IM}(G,\mathcal{M})$, each vertex of $I_\mathcal{M}(H)$ is at distance at most 2 from some vertex of $\bigcup_{x\in X \cap Y} M_x$.
    Since $|X\cap Y|\le 2$ and $\textbf{IM}(G,\mathcal{M})[I_\mathcal{M}(H)]$ is connected, it now follows that $I_\mathcal{M}(H)$ has weak diameter at most $(2+2)+1+(2+2)=9$ in $\textbf{IM}(G,\mathcal{M})$.

    Hence $(G, \mathcal{M})$ is indeed a $(11,9)$-impression of $(G,\mathcal{H})$, as desired.
\end{proof}

Let $G$ be a finite planar graph and $\mathcal{H}$ a spanning collection of connected vertex sets in $G$, each of which contains a vertex incident to the outer face of $G$.
For a collection disjoint connected sets $\mathcal{B}$ of $G$, we say that $(G,\mathcal{B})$ is a \emph{partial $(x,\infty)$-impression} of $(G,\mathcal{H})$ if for every $B\in \mathcal{B}$, $I_\mathcal{H}(B)$ has weak diameter at most $x$ in $\textbf{RIG}(G,\mathcal{H})$.

We aim to find a $(x,y)$-impression $(G,\mathcal{M})$ of $(G,\mathcal{H})$ in which each $M\in \mathcal{M}$ contains a vertex incident to the outer face of $G$.
In light of \cref{lem:outerplanar}, we can find such a $(x,y)$-impression by first finding such a $(x', \infty)$-impression and then refining this impression.
The following theorem comes close to giving us such a $(x', \infty)$-impression (not every $M\in \mathcal{M}$ contains a vertex incident to the outer face of $G$, but they are at distance at most $1$ from such a $M'\in \mathcal{M}$ in $\textbf{IM}(G,\mathcal{M})$).
The exact technical statement of this theorem is tailored to allow for an inductive argument on extending partial $(x',\infty)$-impressions.

\begin{theorem}\label{thm:outerstringinduct}
    Let $G$ be a finite planar graph and $\mathcal{H}$ be a spanning collection of connected vertex sets in $G$, each of which contains a vertex incident to the outer face of $G$.
    Let $(G, \mathcal{B})$ be a partial $(30,\infty)$-impression of $(G,\mathcal{H})$ in which for every $B\in \mathcal{B}$, either $B$ contains a vertex incident to the outer face of $G$ and $I_\mathcal{H}(B)$ has weak diameter at most 10 in $\textbf{RIG}(G,\mathcal{H})$, or $B$ touches some $B' \in \mathcal{B}$ where $B'$ contains a vertex incident to the outer face of $G$, and
    for every maximal connected vertex set $K$ of $G\backslash (\bigcup_{B\in \mathcal{B}} B)$, either
    \begin{enumerate}
        \item\label{item:1} $I_\mathcal{H}(K)$ has weak diameter at most $10$ in $\textbf{RIG}(G,\mathcal{H})$ and contains a vertex incident to the outer face of $G$, or $I_\mathcal{H}(K)$ has weak diameter at most $30$ in $\textbf{RIG}(G,\mathcal{H})$ and touches some $B\in \mathcal{B}$ that does contain a vertex incident incident to the outer face of $G$, or

        \item\label{item:2} $K$ contains a vertex incident to the outer face of $G$ and touches no $B\in \mathcal{B}$, or 

        \item\label{item:3} $K$ contains a vertex incident to the outer face of $G$ and touches a unique $B\in \mathcal{B}$, or 

        \item\label{item:4} $K$ contains a vertex incident to the outer face of $G$ and there exists distinct $B_1,B_2\in \mathcal{B}$ satisfying the following
        \begin{itemize}
            \item $K$ touches $B_1$ and $B_2$ and no other $B\in \mathcal{B}\backslash \{B_1,B_2\}$,
            \item $B_1$ and $B_2$ both contain a vertex incident to the outer face of $G$, and
            \item $B_1$ and $B_2$ both contain a vertex incident to a common face $F$ distinct from the outer-face.
        \end{itemize}
    \end{enumerate}
    Then there is a $(30, \infty)$-impression $(G, \mathcal{M})$ of $(G,\mathcal{H})$ such that for every $M\in \mathcal{M}$, either $M$ contains a vertex incident to the outer face of $G$ and $I_\mathcal{H}(B)$ has weak diameter at most 10 in $\textbf{RIG}(G,\mathcal{H})$, or $M$ touches some $M' \in \mathcal{M}$ where $M'$ contains a vertex incident to the outer face of $G$.
\end{theorem}

\begin{proof}
    We shall argue inductively on $|V(G)\backslash (\bigcup_{B\in \mathcal{B}} B)|$.
    Clearly the theorem holds if we have $V(G)\backslash (\bigcup_{B\in \mathcal{B}} B) = \emptyset$. So, we may assume that there is some connected component of $G\backslash (\bigcup_{B\in \mathcal{B}} B)$, and we let $K$ be its vertex set.
    
    Suppose first that $K$ satisfies \cref{item:1}.
    Let $B^*=K$ and $\mathcal{B}^*= \mathcal{B} \cup \{B^*\}$.
    Clearly, $(G,\mathcal{B}^*)$ satisfies the inductive hypothesis and $|V(G)\backslash (\bigcup_{B\in \mathcal{B}^*} B)|< |V(G)\backslash (\bigcup_{B\in \mathcal{B}} B)|$, so the theorem follows.

    Now, suppose that $K$ satisfies \cref{item:2}.
    Then simply choose some $b\in K$ incident to the outer face of $G$, let $B^*=\{b\}$ and $\mathcal{B}^*= \mathcal{B} \cup \{B^*\}$.
    Clearly, $(G,\mathcal{B}^*)$ satisfies the inductive hypothesis and $|V(G)\backslash (\bigcup_{B\in \mathcal{B}^*} B)|< |V(G)\backslash (\bigcup_{B\in \mathcal{B}} B)|$, so the theorem follows.
    
    Now, suppose that $K$ satisfies \cref{item:3}.
    Let $B$ be the unique set in $\mathcal{B}$ that $K$ touches.
    Then $I_\mathcal{H}(B)$ has weak diameter at most 10 in $\textbf{RIG}(G,\mathcal{H})$.
    As $K$ contains a vertex incident to the outer face of $G$, there exists intersecting (but possibly non-distinct) $H_1,H_2\in \mathcal{H}$ such that $K\cap (H_1\cup H_2)$ contains a  vertex $b$ incident to the outer face of $G$, and $K\cap (H_1\cup H_2)$ touches $B$.
    Let $B^*$ be the vertex set of a shortest path between $b$ and $N_G(B)$ in $G[K\cap (V(H_1)\cup V(H_2))]$, and let $\mathcal{B}^*= \mathcal{B} \cup \{B^*\}$.
    Then $I_\mathcal{H}(B^*)$ has weak diameter at most 3 in $\textbf{RIG}(G,\mathcal{H})$, and clearly $B^*$ touches $B$.
    Let $K'$ be the vertex set of a connected component of $G[K]\backslash B^*$.
    If $K'$ contains a vertex incident to the outer face of $G$, then $K'$ satisfies either \cref{item:3} or \cref{item:4}.
    Otherwise, $K'$ contains no vertex incident to the outer face of $G$. As $K'$ touches none of $\mathcal{B^*}\backslash \{B,B^*\}$ (but at least one of $B,B^*$) and for every $H\in \mathcal{H}$ that intersects $K'$, we have that $H$ also intersects at least one of $B$ or $B^*$, we have that $I_\mathcal{H}(K')$ has weak diameter at most $10+3=13$ in $\textbf{RIG}(G,\mathcal{H})$.
    Thus, $K'$ then satisfies \cref{item:1}.
    So, $(G,\mathcal{B}^*)$ satisfies the inductive hypothesis and $|V(G)\backslash (\bigcup_{B\in \mathcal{B}^*} B)|< |V(G)\backslash (\bigcup_{B\in \mathcal{B}} B)|$, so the theorem follows.

    Lastly, we can now assume that $K$ satisfies \cref{item:4}.
    Let $B_1, B_2, F$ be as in \cref{item:4}.
    Note that the intersection of $K$ with the vertices of the outer face of $G$ induces a connected graph since $B_1$ and $B_2$ both contain a vertex on the outer face of $G$ and the face $F$ (which is distinct from the outer face). Let $Q$ be the vertex set of this connected induced subgraph.
    Let $p_1p_2\ldots p_t$ be the walk in $G[Q]$ along the outer face of $G$ (and containing all of) $Q$, where $p_1$ touches $B_1$ and $p_t$ touches $B_2$.
    Choose $1\le \ell \le t$ maximum such that that there exists a path with vertex set $L$ between $p_\ell$ and $N_G(B_1)$ such that $L$ is contained in $K\cap (H_\ell\cup H'_\ell)$ for some intersecting (but not necessarily distinct) $H_\ell,H'_\ell \in \mathcal{H}$.
    Similarly, choose $1 \le r \le t$ minimum such that that there exists a path with vertex set $R$ between $p_r$ and $N_G(B_2)$ such that $R$ is contained in $K\cap (H_r\cup H'_r)$ for some intersecting (but not necessarily distinct) $H_r,H'_r \in \mathcal{H}$.
    Note that such $\ell, r$ exist since $K$ touches exactly $B_1,B_2$ from $\mathcal{B}$, and $K$ (and every $H\in \mathcal{H}$) contains a vertex incident to the outer face of $G$.
    Let $L^*$ be the vertex set of the connected component of $G\left[ K \cap \left( \bigcup_{  H\in I_\mathcal{H}(L)}H \right) \right]$ containing $L$, and similarly, let $R^*$ be the vertex set of the connected component of $G\left[ K \cap \left( \bigcup_{  H\in I_\mathcal{H}(R)}H \right) \right]$ containing $R$.
    Clearly $I_\mathcal{H}(L^*)$ and $I_\mathcal{H}(R^*)$ both have weak diameter at most 5 in $\textbf{RIG}(G,\mathcal{H})$.
    See \cref{fig:outerstring} for an illustration (in a case where $L^*$ and $R^*$ do not intersect).
    We split into two cases depending on whether or not $L^*$ and $R^*$ intersect.

    \begin{figure}
    \centering
    \includegraphics[width=0.82\linewidth]{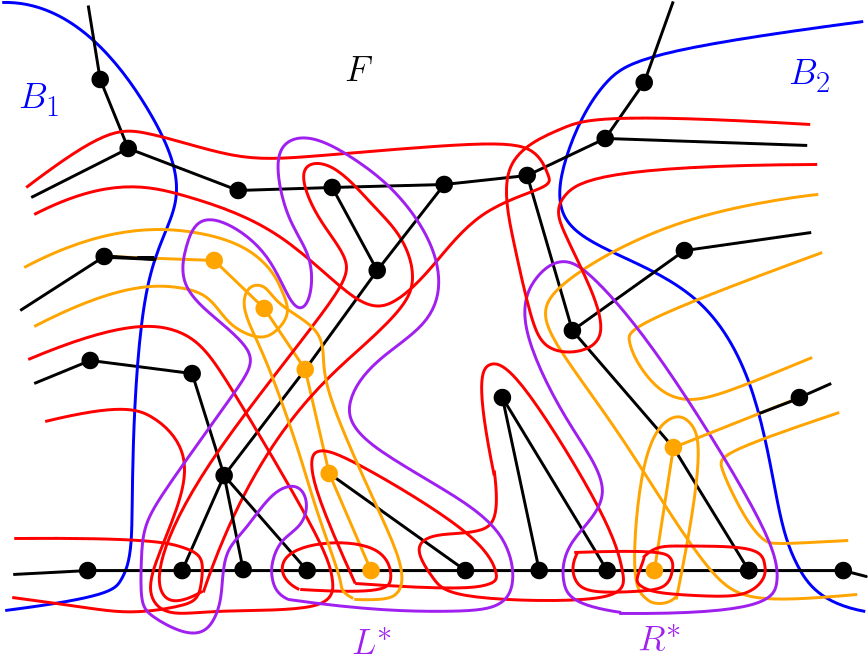}
    \caption{An illustration of the constructed connected sets $L^*$ and $R^*$ (coloured purple) in a case where they do not intersect.
    Both $B_1$ and $B_2$ are coloured blue.
    The connected sets in $\mathcal{H}$ are coloured red or orange, with the orange ones being $H_\ell, H_\ell', H_r,H_r'$.
    The two orange paths indicate $L$ and $R$.
    }
    \label{fig:outerstring}
\end{figure}

    Suppose first that $L^*$ and $R^*$ do intersect.
    Then let $B^*=L^* \cup R^*$, and let $\mathcal{B}^*= \mathcal{B} \cup \{B^*\}$.
    Clearly $I_\mathcal{H}(B^*)$ has weak diameter at most $5+5=10$ in $\textbf{RIG}(G,\mathcal{H})$.
    Note that $B^*$ touches both $B_1$ and $B_2$.
    Consider the vertex set $K'$ of a connected component of $G[K]\backslash B^*$.
    If $K'$ contains a vertex incident to the outer face of $G$, then $K'$ clearly satisfies \cref{item:3} or \cref{item:4}.
    So suppose that $K'$ contains no vertex incident to the outer face of $G$.
    Then $K'$ touches at least one of $B_1,B_2,B^*$ and for every $H\in I_\mathcal{H}(K')$, we have that $H$ intersects at least one of $B_1,B_2,B^*$.
    Since $B^*$ touches both $B_1$ and $B_2$, it follows that $I_\mathcal{H}(B^*)$ has weak diameter at most $10+10+10=30$ in $\textbf{RIG}(G,\mathcal{H})$.
    Thus, $K'$ satisfies \cref{item:1}.
    So, $(G,\mathcal{B}^*)$ satisfies the inductive hypothesis and $|V(G)\backslash (\bigcup_{B\in \mathcal{B}^*} B)|< |V(G)\backslash (\bigcup_{B\in \mathcal{B}} B)|$, so the theorem follows.

    We may now assume that $L^*$ and $R^*$ do not intersect.
    This in particular implies that $\ell < r$.
    Let $\mathcal{B}^*= \mathcal{B} \cup \{L^*, R^*\}$.
    Consider the vertex set $K'$ of a connected component of $G[K]\backslash (L^* \cup R^*)$.
    Suppose that $K'$ does not contain a vertex incident to the outer face of $G$.
    Then $K$ touches at least one of $B_1,B_2,L^*,R^*$, and for every $H\in I_\mathcal{H}(K')$, we have that $H$ intersects at least one of $B_1,B_2,L^*,R^*$.
    Since $I_\mathcal{H}(B_1), I_\mathcal{H}(B_2)$ both have weak diameter at most 10 in $\textbf{RIG}(G,\mathcal{H})$ and $I_\mathcal{H}(L^*), I_\mathcal{H}(R^*)$ both have weak diameter at most 5 in $\textbf{RIG}(G,\mathcal{H})$, it follows that $I_\mathcal{H}(K')$ has weak diameter at most 30 in $\textbf{RIG}(G,\mathcal{H})$.
    Thus in this case, $K'$ satisfies \cref{item:1}.
    So, we may now assume that $K'$ does contain a vertex incident to the outer face of $G$.
    If $K'$ does not touch both $L^*$ and $R^*$, then as before, $K'$ satisfies either \cref{item:3} or \cref{item:4}.
    So we assume now that $K'$ touches both $L^*$ and $R^*$.
    Showing that $K'$ satisfies \cref{item:4} will complete the proof.

    Let $Q'$ be the vertices $p_i$ of the walk $Q$ with $\ell < i <r$. Note that $Q'$ is non-empty since $L^*$ and $R^*$ do not intersect (and $L^*$ contains both $p_i$ and $p_{i+1}$).
    Let $\mathcal{H}_{Q'}= I_\mathcal{H}(Q')$.
    Observe that by our choice of $L$ and $R$, there is no intersecting pair $H_1\in \mathcal{H}_{Q'}$, $H_2\in \mathcal{H}$ such that there exists a path in $G[K \cap (H_1 \cup H_2)]$ between $Q'$ and $N_G(B_1\cup B_2)$.
    By our choice of $L^*$ and $R^*$, we therefore have that 
    $I_\mathcal{H}(K') \subseteq \mathcal{H}_{Q'}$.
    In particular, $K'$ does not touch $B_1$ or $B_2$.
    It now follows that there is some face $F$ of $G$ distinct from the outer face containing a vertex of $L^*$ and a vertex of $R^*$.
    Thus $K'$ satisfies \cref{item:4}.
    
    So, $(G,\mathcal{B}^*)$ satisfies the inductive hypothesis and $|V(G)\backslash (\bigcup_{B\in \mathcal{B}^*} B)|< |V(G)\backslash (\bigcup_{B\in \mathcal{B}} B)|$, so the theorem follows.
\end{proof}

Now we quickly refine the previous theorem so that every graph in $\mathcal{M}$ contains a vertex incident to the outer face of $G$.

\begin{lemma}\label{lem:outerstring}
    Let $G$ be a finite planar graph and $\mathcal{H}$ a spanning collection of connected vertex sets of $G$, each of which contains a vertex incident to the outer face of $G$.
    Then there is a $(70, \infty)$-impression $(G, \mathcal{M})$ of $(G,\mathcal{H})$ such that for every $M\in \mathcal{M}$, $M$ contains a vertex incident to the outer face of $G$.
\end{lemma}

\begin{proof}
    By \cref{thm:outerstringinduct} (applied with $\mathcal{B}=\emptyset$), there is a $(30, \infty)$-impression $(G, \mathcal{M}')$ of $(G,\mathcal{H})$ such that for every $M\in \mathcal{M}'$, either $M$ contains a vertex incident to the outer-face of $G$ and $I_\mathcal{H}(B)$ has weak diameter at most 10 in $\textbf{RIG}(G,\mathcal{H})$, or $M$ touches some $M' \in \mathcal{M}'$ where $M'$ contains a vertex incident to the outer face of $G$.

    Let $\mathcal{M}_1$ be the elements $M$ of $\mathcal{M}'$ that contain a vertex incident to the outer face of $G$, and let $\mathcal{M}_2$ be the elements $M$ of $\mathcal{M}'$ that do not contain a vertex incident to the outer face of $G$.
    Then $\mathcal{M}' = \mathcal{M}_1 \cup \mathcal{M}_2$.
    Choose some function $f:\mathcal{M}_2 \to \mathcal{M}_1$ such that $M$ touches $f(M)$ for every $M\in \mathcal{M}_2$.
    Let $\mathcal{M}=\{M\cup f^{-1}(M) : M\in \mathcal{M}_1)\}$.
    Then for every $M\in \mathcal{M}$, we have that $I_{\mathcal{H}}(M)$ has weak diameter at most $30+10+30= 70$ in $G$.
    Then $(G, \mathcal{M})$ is a $(70, \infty)$-impression of $(G,\mathcal{H})$ such that for every $M\in \mathcal{M}$, $M$ contains a vertex incident to the outer face of $G$, as desired.
\end{proof}

We can now quickly prove \cref{thm:outerstring} using \cref{lem:outerstring} and \cref{lem:outerplanar}.
We remark that in our proof of \cref{thm:stringmainfinite}, it actually turns out to be more convenient to use \cref{lem:outerstring} and \cref{lem:outerplanar} separately as below instead of \cref{thm:outerstring}.
This is essentially because we end up needing to apply \cref{lem:outerplanar} another time later anyway.
However for the sake of completeness for outerstring graphs, we still complete the proof of \cref{thm:outerstring} here.

\begin{proof}[Proof of \cref{thm:outerstring}.]
    By \cref{lem:outerstring}, there is a $(70, \infty)$-impression $(G, \mathcal{M}')$ of $(G,\mathcal{H})$ such that for every $M\in \mathcal{M}'$, $M$ contains a vertex incident to the outer face of $G$.
    Let $G'=\textbf{IM}(G,\mathcal{M}')$.
    Then $G'$ is outerplanar.
    For each $H \in \mathcal{H}$, let $H'$ be the connected vertex set of $G'$ that is exactly the union of all $M'\in \mathcal{M}'$ such that $M'$ intersects $H$.
    Let $\mathcal{H}'= \{H' : H \in \mathcal{H}\}$.

    By \cref{lem:outerplanar}, there is a $(11,9)$-impression $(G', \mathcal{M}^*)$ of $(G',\mathcal{H}')$.
    Let $\mathcal{M} = \{\bigcup_{M'\in M^*} M' : M^*\in \mathcal{M}\}$.
    Clearly for every $H\in \mathcal{H}$, $I_\mathcal{M}(H)$ has weak diameter at most $9$ in $\textbf{IM}(G,\mathcal{M})$.
    Also, for every $M\in \mathcal{M}$, we have that $I_{\mathcal{H}}(M)$ has weak diameter at most $11\cdot 70 = 770$ in $\textbf{RIG}(G,\mathcal{H})$.
    
    Hence, $(G,\mathcal{M})$ is a $(770,9)$-impression of $(G,\mathcal{H})$, as desired.
\end{proof}

\section{Encasings}\label{sec:encase}

The key to carrying out our inductive proof of \cref{thm:stringmainfinite} in \cref{sec:string} will be that we can find what we shall call a $(a,b,c,d)$-encasing.
The aim of this section is to find these encasings.
We show that we can find $(a,b,c,d)$-encasing by first finding a $(a,\infty,\infty,\infty)$-encasing, then refining it to a $(a,\infty,c,\infty)$-encasing, then further refining it to a $(a,\infty,c,d)$-encasing, and then finally refining it to a $(a,b,c,d)$-encasing.
To ease into things, we delay the full definition of an $(a,b,c,d)$-encasing for now, and start simply with $(a,\infty,\infty,\infty)$-encasings.

Let $G$ be a finite planar graph and let $\mathcal{H}$ be a spanning collection of connected vertex sets in $G$.
Let $\mathcal{B}$ be a collection of disjoint connected vertex sets in $G$.
We say that $\mathcal{B}$ \emph{encircles} $(G,\mathcal{H})$ if every $B\in \mathcal{B}$ contains a vertex incident to the outer face of $G$, and every $H\in \mathcal{H}$ that contains a vertex incident to the outer face of $G$ is contained in $\bigcup_{B\in \mathcal{B}}B$.
We say that $\mathcal{B}$ is a \emph{$(a,\infty,\infty,\infty)$-encasing} of $(G,\mathcal{H})$ if $\mathcal{B}$ encircles $(G,\mathcal{H})$, and
\begin{itemize}
    \item for every $B\in \mathcal{B}$, we have that $I_{\mathcal{H}}(B)$ has weak diameter at most $a$ in $\textbf{RIG}(G,\mathcal{H})$.
\end{itemize}

We begin by finding a $(a,\infty,\infty,\infty)$-encasing. The lemma has an extra condition on the vertices contained in $\bigcup_{B\in \mathcal{B}}B$ that we will need later.

\begin{lemma}\label{lem:encase1}
    Let $G$ be a finite planar graph and $\mathcal{H}$ a spanning collection of connected vertex sets of $G$.
    Then there is a $(210,\infty,\infty,\infty)$-encasing $\mathcal{B}$ of $(G, \mathcal{H})$ with $\bigcup_{B\in \mathcal{B}} B = \bigcup_{H\in \mathcal{H}_0 \cup \mathcal{H}_1} H$, where $\mathcal{H}_0$ are the elements of $\mathcal{H}$ that contain a vertex incident to the outer face of $G$ and $\mathcal{H}_1$ are the elements of $\mathcal{H}\backslash \mathcal{H}_0$ that intersect a element of $\mathcal{H}_0$.
\end{lemma}

\begin{proof}
    Let $G^*$ be the subgraph of $G$ with edges $\bigcup_{H\in \mathcal{H}_0 \cup \mathcal{H}_1} E(G[H])$ and no isolated vertices.
    Note that $G^*$ has the same outer face as $G$.
    For each $H\in \mathcal{H}_1$, choose some $H'\in \mathcal{H}_0$ that intersects $H$ and let $H^* = H \cup H'$.
    Let $\mathcal{H}^*=\mathcal{H}_0 \cup \{H^* : H\in \mathcal{H}_1\}$.
    Then every vertex on the outer face of $G$ is contained in some element of $\mathcal{H}^*$.
    By \cref{lem:outerstring}, there is a $(70,\infty)$-impression $(G^*,\mathcal{B})$ of $(G^*, \mathcal{H}^*)$ such that every vertex of $G^*$ incident to the outer face of $G^*$ is contained in some $B\in \mathcal{B}$.
    In particular, every vertex of $G$ incident to the outer face of $G$ is contained in some $B\in \mathcal{B}$.
    Furthermore, observe that as $I_{\mathcal{H}^*}(B)$ has weak diameter at most 70 in $\textbf{RIG}(G^*,\mathcal{H}^*)$ for each $B\in \mathcal{B}$, we also have that $I_{\mathcal{H}^*}(B)$ has weak diameter at most $3\cdot 70 = 210$ in $\textbf{RIG}(G,\mathcal{H})$ for each $B\in \mathcal{B}$.
\end{proof}

Let $G$ be a finite planar graph, $\mathcal{H}$ a spanning collection of connected vertex sets of $G$, and let $\mathcal{B}$ be a collection of disjoint connected sets of $G$ that encircles $(G,\mathcal{H})$.
Let $\mathcal{F}$ be the collection of non-empty vertex sets of $G$ that are maximal with respect to being contained in the interior of a face of $G[\bigcup_{B\in \mathcal{B}}B]$.
We call these \emph{facial} vertex sets.
We remark that these facial vertex sets are not necessarily connected sets of $G$.
Note that $\mathcal{B} \cup \mathcal{F}$ is a partition of $V(G)$.
We denote by $G(\mathcal{B})$ the graph obtained from $\textbf{IM}(G,\mathcal{B})$ by for every facial vertex set $\mathcal{F}$ of $G[\bigcup_{B\in \mathcal{B}}B]$, we add $F$ as a vertex and make it adjacent to exactly the elements $B\in \mathcal{B}$ that contain a vertex incident to the face of $G[\bigcup_{B\in \mathcal{B}}B]$ that contains $F$ (we remark that $B$ and $F$ need not touch).

We say that some $H\in \mathcal{H}$ is $d$-caged by $\mathcal{B}$ if for every $u,v\in H$, there exists some sequence of vertex sets $A_0,  \ldots , A_t$ with $t\le d$ such that
\begin{itemize}
    \item each $A_i$ is either a single element of $\mathcal{B}$ or a single element of $\mathcal{F}$ or the union of a subcollection $\mathcal{B}'\subseteq \mathcal{B}$ such that some connected set $H^*\subseteq H$ of $G$ intersects each member of $\mathcal{B}'$ and is contained in $G[\bigcup_{B\in \mathcal{B}'}B]$,
    \item $u\in A_0$, 
    \item $v\in A_t$, and
    \item for every $1\le i <t$, we have that either $H$ intersects both $A_{i-1}$ and $A_i$, or neither $A_{i-1}$ or $A_i$ is an element of $\mathcal{F}$ and they touch.
\end{itemize}
See \cref{fig:caged} for an illustration of some $7$-caged $H$.
We say that $\mathcal{B}$ is a $d$-cage of $(G,\mathcal{H)}$ if every $H\in \mathcal{H}$ is $d$-caged by $\mathcal{B}$.

\begin{figure}
    \centering
    \includegraphics[width=0.75\linewidth]{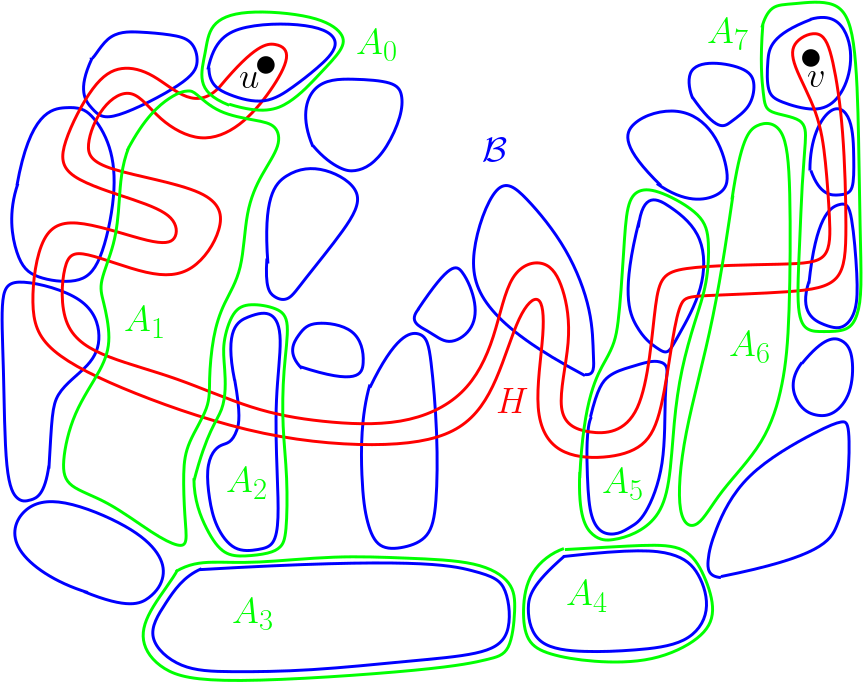}
    \caption{An illustration of some $H\in \mathcal{H}$ (coloured red) that is 7-caged by $\mathcal{B}$ (coloured blue).
    For $u,v\in H$, the vertex sets $A_0,\ldots , A_7$ (coloured green) are as in the definition of being 7-caged.
    As an example illustrating the end of the proof of \cref{lem:encase2}, we would have that $A_0=A_u$, $A_1\in \mathcal{F}_u$, $A_2,A_5\in \mathcal{A}_{u,v}$, $\{A_3,A_4\}=\mathcal{V}_{(X,Y)}$, $A_6\in \mathcal{F}_v$, and $A_7=A_v$.
    }
    \label{fig:caged}
\end{figure}

The motivation behind the notion of a $d$-cage is that it will allow us to ``cut up'' connected sets $H$ in such a way that each part is in some sense still close via the encircling $\mathcal{B}$.
Being able to do this is crucial to the inductive argument of \cref{thm:stringmainfinite}.

We say that $\mathcal{B}$ is a \emph{$(a,\infty,\infty,d)$-encasing} of $(G,\mathcal{H})$ if $\mathcal{B}$ encircles $(G, \mathcal{H})$, and
\begin{itemize}
            \item for every $B\in \mathcal{B}$, we have that $I_{\mathcal{H}}(B)$ has weak diameter at most $a$ in $\textbf{RIG}(G,\mathcal{H})$, and
            \item $\mathcal{B}$ is a $d$-cage of $(G, \mathcal{H})$.
\end{itemize}
In other words, $\mathcal{B}$ is a $(a,\infty,\infty,d)$-encasing of $(G,\mathcal{H})$ if it is both a $(a,\infty,\infty,\infty)$-encasing and a $d$-cage of $(G,\mathcal{H})$.

We will use our $(a,\infty,\infty,\infty)$-encasing to get a $(a,\infty,\infty,d)$-encasing.
However, we again need this $(a,\infty,\infty,d)$-encasing to have some stronger structure on the vertices of $G$ that $\bigcup_{B\in \mathcal{B}}B$ contains.
Let $\mathcal{H}_0$ be the elements of $\mathcal{H}$ that contain a vertex incident to the outer face of $G$, let $\mathcal{H}_1$ be the elements of $\mathcal{H}\backslash \mathcal{H}_0$ that intersect a element of $\mathcal{H}_0$, and similarly, let $\mathcal{H}_2$ be the elements of $\mathcal{H}\backslash (\mathcal{H}_0 \cup \mathcal{H}_1)$ that intersect a element of $\mathcal{H}_1$.
We say that a $(a,\infty,\infty,d)$-encasing $\mathcal{B}$ of $(G,\mathcal{H})$ is \emph{enforced} if $\bigcup_{H\in \mathcal{H}_0 \cup \mathcal{H}_1} H \subseteq \bigcup_{B\in \mathcal{B}} B \subseteq \bigcup_{H\in \mathcal{H}_0 \cup \mathcal{H}_1 \cup \mathcal{H}_2} H$, and for every $v\in \left(\bigcup_{B\in \mathcal{B}} B \right) \backslash \left( \bigcup_{H\in \mathcal{H}_0 \cup \mathcal{H}_1} H \right)$, 
there is some connected set $H'\subseteq H\in \mathcal{H}_2$ of $G$ that contains $v$, is contained in $\bigcup_{B\in \mathcal{B}} B$, and intersects $\bigcup_{H\in \mathcal{H}_0 \cup \mathcal{H}_1} H$.

\begin{lemma}\label{lem:encase2}
    Let $G$ be a finite planar graph and $\mathcal{H}$ a spanning collection of connected vertex sets of $G$.
    Then there is an enforced $(840,\infty,\infty,7)$-encasing $\mathcal{B}$ of $(G, \mathcal{H})$.
\end{lemma}

\begin{proof}
    We define $\mathcal{H}_0, \mathcal{H}_1, \mathcal{H}_2$ as above in the definition of an enforced $(a,\infty,\infty,d)$-encasing of $(G,\mathcal{H})$.
    By \cref{lem:encase1}, there is a $(210,\infty,\infty,\infty)$-encasing $\mathcal{B}$ of $(G, \mathcal{H})$ with $\bigcup_{B\in \mathcal{B}} B = \bigcup_{H\in \mathcal{H}_0 \cup \mathcal{H}_1} H$.
    We will adjust this encasing to obtain the desired enforced $(840,\infty,\infty,7)$-encasing.

    Let $G^*=G(\mathcal{B})$ as above, where $\mathcal{F}$ are the additional facial vertex sets.
    Choose some $B_0\in \mathcal{B}$.
    We say that a separation $(X,Y)$ of $G^*$ is \emph{1-clean} if $B_0\in X$, $X\cup Y = V(G^*)$, $|X\cap Y|=1$, and subject to this, there is no other such $(X',Y')$ with $X'\cap Y' = X \cap Y$ and $|Y'|>|Y|$.
    A separation $(X,Y)$ of $G^*$ is \emph{2-clean} if $B_0\in X$, $X\cup Y = V(G^*)$, $X\cap Y$ has two elements and they are adjacent in $G^*$, and there is no 1-clean separation $(X^*,Y^*)$ with $X^*\cap Y^* \subset X \cap Y$ and $Y^*\subseteq Y$, and there is no other such separation $(X',Y')$ of $G^*$ with $X'\cap Y' = X \cap Y$ and $|Y'|>|Y|$.
    By abuse of notation, we shall also call $(\{B_0\}, V(G^*))$ a 1-clean separation.
    A separation $(X,Y)$ of $G^*$ is \emph{clean} if it is 1-clean or 2-clean.
    Observe that if $(X,Y)$ is a clean separation, then $X\cap Y \subseteq \mathcal{B}$.

    For a clean separation $(X,Y)$ of $G^*$, we let $R_{(X,Y)}$ be the vertices $R$ of $G^*$ in $Y\backslash \mathcal{F}$ (so $R$ is one of the vertex sets $Y\backslash \mathcal{F}=Y\cap \mathcal{B}$) such that there exists some $H\in \mathcal{H}$ that intersects both $R$ and $\bigcup_{A\in X\cap Y} A$.
    We say that a clean separation $(X',Y')$ of $G^*$ is a \emph{child} of another clean separation $(X,Y)$ if $Y'\subseteq Y$, $R_{(X,Y)} \cap Y' =\emptyset$, and subject to this, there is no other such clean separation $(X^*,Y^*)$ with $Y^*$ containing $Y'$ as a proper subset.
    We say that a clean separation $(X',Y')$ is a \emph{descendant} of another clean separation $(X,Y)$ if it there is a sequence of clean separations $(X_1,Y_1) , \ldots (X_t, Y_t)$ with $(X_1,Y_1)=(X,Y)$, $(X_t, Y_t)=(X',Y')$ and for each $1 \le i <t$, $(X_{i+1}, Y_{i+1})$ is a child of $(X_i, Y_i)$.
    Let $\mathcal{T}$ be the collection of clean separations consisting of $(\{B_0\},V(G^*))$ and all of its descendants.
    We let $\mathcal{T}_1$ be the 1-clean separations of $\mathcal{T}$, and $\mathcal{T}_2$ be the 2-clean separations of $\mathcal{T}$.

    \begin{figure}
    \centering
    \includegraphics[width=0.7\linewidth]{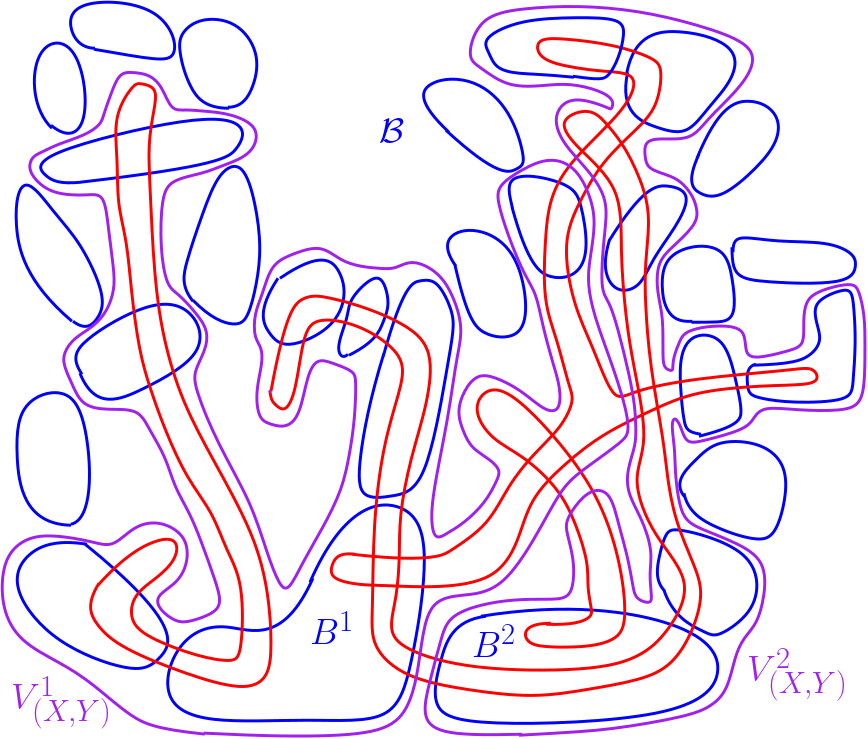}
    \caption{A choice of $V^1_{(X,Y)},V^2_{(X,Y)}$ (coloured purple) as in the proof of \cref{lem:encase2}.
    Elements of $\mathcal{B}$ are coloured blue and elements of $\mathcal{H}$ are coloured red.
    }
    \label{fig:encase2}
\end{figure}
    
    For a clean separation $(X,Y)$ of $G^*$, we let $H_{(X,Y)}$ be the vertices $v$ of $G$ contained in  some vertex set of $\mathcal{F}\cap Y$ such that there exists some $H\in \mathcal{H}_2$ containing both $v$ and intersecting $\bigcup_{A\in X\cap Y} A$.
    If $(X,Y)$ is 1-clean, 
    then we let
    $V_{(X,Y)} = H_{(X,Y)} \cup \bigcup_{R\in R_{(X,Y)}} R$.
    Note that in this case, $V_{(X,Y)}$ is a connected set of $G$ and $I_{\mathcal{H}}(V_{X,Y})$ has weak diameter at most $3\cdot 210 =630$ in $\textbf{RIG}(G,\mathcal{H})$.
    Furthermore, if 
    $v\in V_{(X,Y)} \backslash \left(\bigcup_{B\in \mathcal{B}} B \right) = V_{(X,Y)} \backslash \left( \bigcup_{H\in \mathcal{H}_0 \cup \mathcal{H}_1} H \right)$, 
    then there exists some connected set $H'\subseteq H\in \mathcal{H}_2$ of $G$ that contains $v$, is contained in $V_{(X,Y)}$, and intersects $\bigcup_{H\in \mathcal{H}_0 \cup \mathcal{H}_1} H$ (or more precisely, the unique vertex set of $X\cap Y$). 
    Otherwise, if $(X,Y)$ is 2-clean with $X\cap Y= \{B^1,B^2\}$, 
    then we choose disjoint $V_{(X,Y)}^1$, $V_{(X,Y)}^2$ such that $B^1\subseteq V_{(X,Y)}^1$, $B^2\subseteq V_{(X,Y)}^2$, 
    $V_{(X,Y)}^1 \cup V_{(X,Y)}^2 = H_{(X,Y)} \cup \bigcup_{R\in R_{(X,Y)}} R$, and both $V_{(X,Y)}^1$ and $V_{(X,Y)}^2$ are connected sets of $G$ (see \cref{fig:encase2} for an illustration).
    Then $I_\mathcal{H}(V_{(X,Y)}^1 \cup V_{(X,Y)}^2)$, and in particular both of $I_\mathcal{H}(V_{(X,Y)}^1)$ and $I_\mathcal{H}(V_{(X,Y)}^2)$ all have weak diameter at most $4 \cdot 210 = 840$ in $\textbf{RIG}(G,\mathcal{H})$.
    Furthermore, if 
    $v\in (V_{(X,Y)}^1 \cup V_{(X,Y)}^2) \backslash \left(\bigcup_{B\in \mathcal{B}} B \right) = (V_{(X,Y)}^1 \cup V_{(X,Y)}^2) \backslash \left( \bigcup_{H\in \mathcal{H}_0 \cup \mathcal{H}_1} H \right)$, 
    then there exists some connected set $H'\subseteq H\in \mathcal{H}_2$ of $G$ that contains $v$, is contained in $V_{(X,Y)}^1 \cup V_{(X,Y)}^2$, and intersects $B^1 \cup B^2 \subseteq \bigcup_{H\in \mathcal{H}_0 \cup \mathcal{H}_1} H$.

    Let $\mathcal{B}'$ be the elements of $\mathcal{B}$ that do not intersect $V_{(X,Y)}$ for any $(X,Y)\in \mathcal{T}_1$ and do not intersect $V_{(X,Y)}^1$ or $V_{(X,Y)}^2$ for any $(X,Y)\in \mathcal{T}_2$.
    At last, let $\mathcal{B}^* = \mathcal{B}' \cup \{V_{(X,Y)} : (X,Y)\in \mathcal{T}_1\}
    \cup \{V_{(X,Y)}^1, V_{(X,Y)}^2  : (X,Y)\in \mathcal{T}_2\}$.

    Clearly, we have that $\bigcup_{H\in \mathcal{H}_0 \cup \mathcal{H}_1} H \subseteq \bigcup_{B\in \mathcal{B}^*} B \subseteq \bigcup_{H\in \mathcal{H}_0 \cup \mathcal{H}_1 \cup \mathcal{H}_2} H$, and
    for every $B\in \mathcal{B}^*$, we have that $I_{\mathcal{H}}(B)$ has weak diameter at most $840$ in $\textbf{RIG}(G,\mathcal{H})$.
    We also have that for every $v\in \left(\bigcup_{B\in \mathcal{B}^*} B \right) \backslash \left( \bigcup_{H\in \mathcal{H}_0 \cup \mathcal{H}_1} H \right)$, 
    there is some connected set $H'\subseteq H\in \mathcal{H}_2$ of $G$ that contains $v$, is contained in $\bigcup_{B\in \mathcal{B}^*} B$, and intersects $\bigcup_{H\in \mathcal{H}_0 \cup \mathcal{H}_1} H$.
    To show that $\mathcal{B}^*$ provides the desired $(840,\infty,\infty,7)$-encasing of $(G, \mathcal{H})$, it now just remains to show that $\mathcal{B}^*$ is a $7$-cage of $(G,\mathcal{H})$.

    If $H\in \mathcal{H}_0\cup \mathcal{H}_1$, then $H$ is 0-caged by $\mathcal{B}^*$ since $H\subseteq \bigcup_{B\in \mathcal{B}^*} B$.
    If $H\in \mathcal{H} \backslash (\mathcal{H}_0\cup \mathcal{H}_1 \cup \mathcal{H}_2)$, then $H$ lies in some face of $G(\mathcal{B})$.
    In particular, $H$ intersects at most two elements of $\mathcal{B}^*$ and furthermore if $H$ does intersect two elements of $\mathcal{B^*}$ then they are touching, so it follows that $H$ is 3-caged by $\mathcal{B}^*$.

    It remains to consider $H\in \mathcal{H}_2$.
    To do so, we must first introduce certain subsets $\mathcal{B}^*_{(X,Y)}$ and $\mathcal{F}^*_{(X,Y)}$ of $V(G(\mathcal{B}^*))$ according to clean separations $(X,Y)$ of $G^*$.
    Fix some $H\in \mathcal{H}_2$.
    For each 1-clean separation $(X,Y)\in \mathcal{T}_1$, let $\mathcal{V}_{(X,Y)} = \{V_{(X,Y)}\}$, and for each 2-clean separation $(X,Y)\in \mathcal{T}_2$, let $\mathcal{V}_{(X,Y)} = \{V^1_{(X,Y)}, V_{(X,Y)}^2\}$.
    Let $\mathcal{F}^*$ be the facial vertex sets of $G(\mathcal{B}^*)$.
    For each clean separation $(X,Y)\in \mathcal{T}$, let $\mathcal{F}^*_{(X,Y)}$ be the elements $F$ of $\mathcal{F}^*$ adjacent in $G(\mathcal{B}^*)$ to a element of $\mathcal{V}_{(X,Y)}$.
    Now, let $\mathcal{B}^*_{(X,Y)} =N_{G(\mathcal{B}^*)} (\mathcal{F}^*_{(X,Y)})$. 

    Consider some $H\in \mathcal{H}_2$.
    By our choice of $\mathcal{B}^*$, we have that there exists some clean separation $(X,Y)\in \mathcal{T}$ such that $H$ is entirely contained in $\bigcup_{B\in \mathcal{B}^*_{(X,Y)}} B \cup \bigcup_{F\in \mathcal{F}^*_{(X,Y)}} F$.
    Observe that $G(\mathcal{B}^*)[\mathcal{B}^*_{(X,Y)} \cup \mathcal{F}^*_{(X,Y)}]$ is a planar graph in which the vertices $\mathcal{B}^*_{(X,Y)}$ are incident to the outer face and every $F\in \mathcal{F}^*_{(X,Y)}$ is adjacent to some $V\in \mathcal{V}_{(X,Y)}$, where $|\mathcal{V}_{(X,Y)}|\le 2$ and if $|\mathcal{V}_{(X,Y)}| = 2$, then the two vertices of $\mathcal{V}_{(X,Y)}$ are adjacent.

    Consider some $u,v\in H$.
    Let $H_u$ be a subset of $H$ that contains $u$ and is maximal with respect to being a connected set of $G$ and being contained in $G[\bigcup_{B\in \mathcal{B}}B]$ (or equivalently $G[\bigcup_{B\in \mathcal{B}^*}B]$).
    Let $A_u$ be the union of all $B\in \mathcal{B}^*$ that intersect $H_u$.
    Let $\mathcal{F}_u$ be elements $F$ of $\mathcal{F}^*$ such that $H$ intersects both $F$ and $A_u$.
    We define $A_v$ and $\mathcal{F}_v$ similarly.
    Crucially, note that if $H$ intersects distinct $F_1,F_2 \in \mathcal{F}^*$, then $H$ intersects some $V\in \mathcal{B}$ incident to the face of $G(\mathcal{B}^*)$ corresponding to $F_1$ that is also either contained in $\mathcal{V}_{(X,Y)}$ or adjacent to some element of $\mathcal{V}_{(X,Y)}$ in $G(\mathcal{B}^*)$.
    Let $\mathcal{A}_{u,v}$ be the collection of unions of subcollections $\mathcal{B}'\subseteq \mathcal{B}^* \backslash \mathcal{V}_{(X,Y)}$ such that some connected set $H^*\subseteq H$ of $G$ intersects each member of $\mathcal{B}'$ and is contained in $G[\bigcup_{B\in \mathcal{B}'}B]$, and some $B\in \mathcal{B'}$ touches some $V\in \mathcal{V}_{(X,Y)}$.
    To see that $H$ is 7-caged by $\mathcal{B}^*$, it is now straight forward to choose the required sequence of vertex sets $A_0, \ldots , A_t$ with $t\le 7$ from the collection $\{A_u\} \cup \mathcal{F}_u \cup (\mathcal{A}_{u,v} \cup  \mathcal{V}_{(X,Y)}) \cup \mathcal{F}_v \cup \{A_v\}$ (here we take $A_0=A_u$, $A_t=A_v$, at most one $A_i$ is contained in $\mathcal{F}_u$ (this would be $A_1$ if so), at most one is contained in $\mathcal{F}_v$ (this would be $A_{t-1}$ if so), at most two are contained in $\mathcal{A}_{u,v}$, and at most two more are contained in $\mathcal{V}_{(X,Y)}$).
    For an illustrative example, see \cref{fig:caged}.

    So, $\mathcal{B}^*$ is indeed a 7-cage of $(G, \mathcal{H})$, and thus $\mathcal{B}^*$ is an enforced $(840,\infty , \infty , 7)$-encasing of $(G,\mathcal{H})$, as desired.
\end{proof}

Let $G$ be a finite planar graph, $\mathcal{H}$ a spanning collection of connected vertex sets of $G$, and let $\mathcal{B}$ be a collection of disjoint connected sets of $G$ such that $\mathcal{B}$ encircles $(G,\mathcal{H})$.
Let $\mathcal{R}$ be the interiors of the internal faces of $G[\bigcup_{B\in \mathcal{B}}B]$.
For $u\in \bigcup_{B\in \mathcal{B}}B$ and $H\in \mathcal{H}$, we say that $u$ is \emph{bounded} by $H$ in $G[\bigcup_{B\in \mathcal{B}}B]$ if either $u\in H$ or $u$ is contained in one of the bounded regions of $\mathbb{R}^2 \backslash (G[H] \cup \bigcup_{R\in \mathcal{R}}R)$.
We say that $U\subseteq \bigcup_{B\in \mathcal{B}}B$ is \emph{bounded} by $H$ in $G[\bigcup_{B\in \mathcal{B}}B]$ if each $u\in U$ is bounded by $H$ in $G[\bigcup_{B\in \mathcal{B}}B]$.
Let $F$ be the vertices of $G$ contained in some face of $G[\bigcup_{B\in \mathcal{B}}B]$.
We say that a vertex $v\in F$ is \emph{$\mathcal{B}$-encroached} by some $H\in \mathcal{H}$ with $v\notin H$ if there exists some connected set $U\subseteq \bigcup_{B\in \mathcal{B}}B$ of $G$ and edges $u_1v_1, \ldots, u_rv_r$ between $F$ and $\bigcup_{B\in \mathcal{B}}B$ that appear consecutively along the boundary of the face of $G[\bigcup_{B\in \mathcal{B}}B]$ corresponding to $F$, such that
\begin{itemize}
    \item $U$ is bounded by $H$ in $G[\bigcup_{B\in \mathcal{B}}B]$,
    \item for each $1\le i <r$, there is a facial walk of $G$ between $u_i$ and $u_{i+1}$ that is contained in $U$,
    \item $u_1,v_1,u_r,v_r\in H$, and
    \item $v\in \{v_1, \ldots , v_{r}\}$.
\end{itemize}
See \cref{fig:encrouched} for an illustrative example.
We further say that some $H'\in \mathcal{H}$ is \emph{$\mathcal{B}$-encroached} by some $H\in \mathcal{H}$ if there is some set $F$ of vertices of $G$ contained in some face of $G[\bigcup_{B\in \mathcal{B}} B]$ and some vertex $v\in H'\cap F$ such that $v$ is $\mathcal{B}$-encroached by $H$.

\begin{figure}
    \centering
    \includegraphics[width=0.55\linewidth]{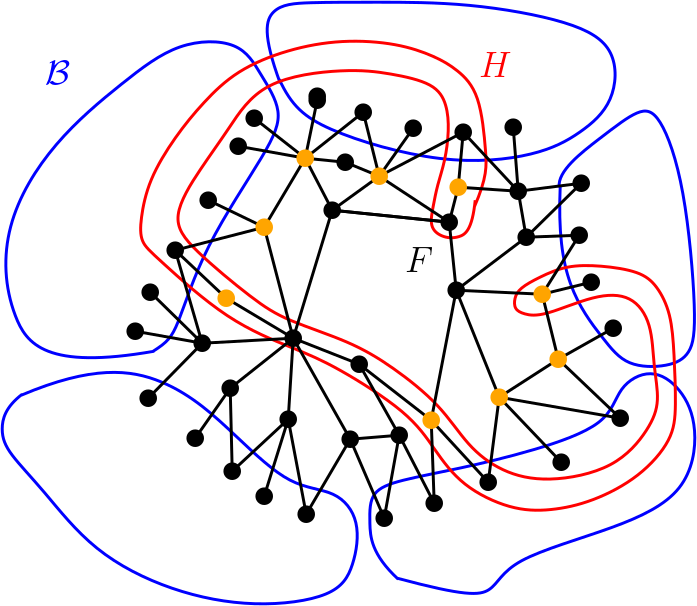}
    \caption{An encircling $\mathcal{B}$ (in blue) and some (red) $H\in \mathcal{H}$.
    The vertices of the facial vertex set $F$ that are encroached by $H$ are coloured orange.
    }
    \label{fig:encrouched}
\end{figure}

Let $\mathcal{H}_0$ be the connected sets in $\mathcal{H}$ that contain a vertex incident to the outer face of $G$.
We say that a collection of disjoint connected sets $\mathcal{B}$ of $G$ that encircles $(G,\mathcal{H})$ is a \emph{$c$-surround} of $(G,\mathcal{H})$ if
\begin{itemize}
    \item for every pair $H_u,H\in \mathcal{H}$ such that $H$ intersects but is not contained in $\bigcup_{B\in \mathcal{B}} B$, and $H_u$ contains a vertex $u \in \bigcup_{B\in \mathcal{B}}B$ that is bounded by $H$ in $G[\bigcup_{B\in \mathcal{B}}B]$, we have that $H_u$ is at distance at most $c-1$ from $H$ in $\textbf{RIG}(G, \mathcal{H} \backslash \mathcal{H}_0)$, and
    \item for every pair $H',H\in \mathcal{H}$ such that $H$ encroaches $H'$, we have that $H'$ is at distance at most $c$ from $H$ in $\textbf{RIG}(G, \mathcal{H} \backslash \mathcal{H}_0)$.
\end{itemize}
Note that since $\mathcal{B}$ encircles $(G,\mathcal{H})$, if $H$ encroaches $H'$, then $H\not\in \mathcal{H}_0$.

The motivation behind the notion of a $c$-surround is that it will allow us to add some (not necessarily yet connected until more edges are added to $G$) sets $H'$ to $\mathcal{H}$ such that $I_{\mathcal{H}}(H')$ has bounded weak diameter in $\textbf{RIG}(G, \mathcal{H} \backslash \mathcal{H}_0)$ and $H'$ is disjoint from $\bigcup_{B\in \mathcal{B}}B$.
Such $H'$ will contain for instance the vertices encroached by a given $H\in \mathcal{H}$ and will be crucial for the inductive argument of \cref{thm:stringmainfinite} in which we remove the vertices of $\bigcup_{B\in \mathcal{B}}B$.

Let $G$ be a finite planar graph and $\mathcal{H}$ a spanning collection of connected vertex sets of $G$.
Let $\mathcal{B}$ be a collection of disjoint connected subgraphs of $G$.
We say that $\mathcal{B}$ is a \emph{$(a,\infty,c,d)$-encasing} of $(G,\mathcal{H})$ if $\mathcal{B}$ encircles $(G,\mathcal{H})$, and
\begin{itemize}
            \item for every $B\in \mathcal{B}$, we have that $I_{\mathcal{H}}(B)$ has weak diameter at most $a$ in $\textbf{RIG}(G,\mathcal{H})$,
            \item $\mathcal{B}$ is a $c$-surround of $(G, \mathcal{H})$, and
            \item $\mathcal{B}$ is a $d$-cage of $(G, \mathcal{H})$.
\end{itemize}

The enforced encasing found in \cref{lem:encase2} is actually a 4-surround (which was the purpose of the technical definition of enforced encasings).

\begin{lemma}\label{lem:encase3}
    Let $G$ be a finite planar graph and $\mathcal{H}$ a spanning collection of connected vertex sets of $G$.
    Then there is a $(840,\infty,4,7)$-encasing $\mathcal{B}$ of $(G, \mathcal{H})$.
\end{lemma}

\begin{proof}
    By \cref{lem:encase2}, there is an enforced $(840,\infty,\infty,7)$-encasing $\mathcal{B}$ of $(G, \mathcal{H})$.
    Let $\mathcal{H}_0,\mathcal{H}_1,\mathcal{H}_2$ be as in the definition of an enforced encasing, and furthermore, let $\mathcal{H}_3$ be the elements of $\mathcal{H}\backslash (\mathcal{H}_0 \cup \mathcal{H}_1 \cup \mathcal{H}_2)$ that intersect an element of $\mathcal{H}_2$.
    We will show that $\mathcal{B}$ is in fact a $(840,\infty,4,7)$-encasing of $(G,\mathcal{H})$.
    
    First, we show the first bullet in the definition of being a 4-surround.
    Consider some pair $H_u,H^*\in \mathcal{H}$ such that $H^*$ intersects but is not contained in $\bigcup_{B\in \mathcal{B}} B$, and $H_u$ contains a vertex $u \in \bigcup_{B\in \mathcal{B}}B$ that is bounded by $H^*$ in $G[\bigcup_{B\in \mathcal{B}}B]$.
    Clearly $H^* \in \mathcal{H}_2 \cup \mathcal{H}_3$, and therefore $u\not\in \bigcup_{H\in {\mathcal{H}_0}}H$.
    In particular, $H_u,H^* \not\in \mathcal{H}_0$.

    Suppose first that $u\in \left( \bigcup_{H\in \mathcal{H}_1} H \right)
    \backslash
    \left( \bigcup_{H\in \mathcal{H}_0} H \right)$.
    Then, there is some $H_u'\in \mathcal{H}_1$ that contains $u$ (possibly $H_u'$ is not distinct from $H_{u}$).
    Observe that $H_u'$ must intersect $H^*$ since $u$ is bounded by $H^*$ in $G[\bigcup_{B\in \mathcal{B}}B]$, $H_u\in \mathcal{H}_1$, and $H^*\in \mathcal{H}_2 \cup \mathcal{H}_3$.
    Therefore, $H_u$ is at distance at most 2 from $H^*$ in $\textbf{RIG}(G, \mathcal{H} \backslash \mathcal{H}_0)$.

    So, we may assume now that $u\in \left( \bigcup_{B\in \mathcal{B}} B \right)
    \backslash
    \left( \bigcup_{H\in \mathcal{H}_0 \cup \mathcal{H}_1} H \right)$.
    Then, by the definition of an enforced encasing, there is some connected set $P_u\subseteq H_u'\in \mathcal{H}_2$ of $G$ that contains $u$, is contained in $\bigcup_{B\in \mathcal{B}} B$, and intersects $\bigcup_{H\in \mathcal{H}_0 \cup \mathcal{H}_1} H$.
    Choose $H_u''\in \mathcal{H}_1$ that intersects $P_u$ (and thus also $H_u'$).
    Similarly to before, $H_u''\cup P_u$ must intersect $H^*$ since $u$ is bounded by $H^*$ in $G[\bigcup_{B\in \mathcal{B}}B]$, $u\in H_u''\cup P_u \subseteq \bigcup_{B\in \mathcal{B}}B$, $H_u''\in \mathcal{H}_1$, and $H^*\in \mathcal{H}_2 \cup \mathcal{H}_3$.
    In particular, $H^*$ intersects at least one of $H_u''$ or $H_u'$.
    Therefore, $H_u$ is at distance at most 3 from $H^*$ in $\textbf{RIG}(G, \mathcal{H} \backslash \mathcal{H}_0)$.
    Thus, $\mathcal{B}$ satisfies the first bullet of being a 4-surround.

    Suppose that some $H'\in \mathcal{H}$ is $\mathcal{B}$-encroached by some $H^*\in \mathcal{H}$.
    Clearly we have that $H',H^*\not\in \mathcal{H}_0$.
    Note that $H^*$ intersects but is not contained in $\bigcup_{B\in \mathcal{B}} B$.
    By definition, there is some set $F$ of vertices of $G$ contained in some face of $G[\bigcup_{B\in \mathcal{B}} B]$ and some vertex $v\in H'\cap F$ such that $v$ is $\mathcal{B}$-encroached by $H^*$.
    Let $u$ be a neighbour of $v$ in $\bigcup_{B\in \mathcal{B}} B$ that is bounded by $H^*$ in $G[\bigcup_{B\in \mathcal{B}}B]$, and let $H_{u}$ be some connected set in $\mathcal{H}$ that contains $\{u,v\}$ (possibly $H_{u}$ is not distinct from $H'$).
    Then, $H'$ is at distance at most $1+3=4$ from $H^*$ in $\textbf{RIG}(G, \mathcal{H} \backslash \mathcal{H}_0)$.

    Thus, $\mathcal{B}$ is a 4-surround of $(G, \mathcal{H})$, and in particular, a $(840,\infty,4,7)$-encasing $\mathcal{B}$ of $(G, \mathcal{H})$ as desired.
\end{proof}

We are finally ready to give the entire definition of a $(a,b,c,d)$-encasing.
We say that $\mathcal{B}$ is a \emph{$(a,b,c,d)$-encasing} of $(G,\mathcal{H})$ if $\mathcal{B}$ encircles $(G,\mathcal{H})$, and
\begin{itemize}
            \item for every $B\in \mathcal{B}$, we have that $I_{\mathcal{H}}(B)$ has weak diameter at most $a$ in $\textbf{RIG}(G,\mathcal{H})$,
            \item for every connected set $H^*\subseteq H\in \mathcal{H}$ of $G$ such that $H^*$ is contained in $\bigcup_{B\in \mathcal{B}}B$, we have that $I_{\mathcal{B}}(H^*)$ has weak diameter at most $b$ in $\textbf{IM}(G,\mathcal{B})$,
            \item $\mathcal{B}$ is a $c$-surround of $(G, \mathcal{H})$, and
            \item $\mathcal{B}$ is a $d$-cage of $(G, \mathcal{H})$.
\end{itemize}

We can at last complete the refinement of the encasing with the use of \cref{lem:outerplanar}.

\begin{lemma}\label{lem:encase4}
    Let $G$ be a finite planar graph and $\mathcal{H}$ a spanning collection of connected vertex sets of $G$.
    Then there is a $(9240,9,4,7)$-encasing $\mathcal{B}$ of $(G, \mathcal{H})$.
\end{lemma}

\begin{proof}
    By \cref{lem:encase3}, there is a $(840,\infty,4,7)$-encasing $\mathcal{B}'$ of $(G, \mathcal{H})$.
    Let $G'=\textbf{IM}(G,\mathcal{B}')$.
    Then $G'$ is outerplanar.
    For every connected set $H^*\subseteq H\in \mathcal{H}$ of $G$ such that $H^*$ is contained in $\bigcup_{B\in \mathcal{B}'}B$, let $H'$ be the connected set of $G'$ whose vertex set is exactly the $B\in \mathcal{B}'$ such that $H^*$ intersects $B$.
    Let $\mathcal{H}'$ be the collection of all of these connected vertex sets of $G'$.

    By \cref{lem:outerplanar}, there is a $(11,9)$-impression $(G', \mathcal{M})$ of $(G',\mathcal{B}')$.
    Let $\mathcal{B} = \{\bigcup_{B\in M} B : M\in \mathcal{M}\}$.
    Clearly for every $H\in \mathcal{H}$, $I_\mathcal{M}(H)$ has weak diameter at most $9$ in $\textbf{IM}(G,\mathcal{M})$.
    Also, for every $M\in \mathcal{M}$, we have that $I_{\mathcal{H}}(M)$ has weak diameter at most $11\cdot 840 =9240$ in $\textbf{RIG}(G,\mathcal{H})$.
    
    Hence, $\mathcal{B}$ is a $(9240,9,4,7)$-encasing of $(G,\mathcal{H})$, as desired.
\end{proof}

\section{String graphs}\label{sec:string}

\cref{lem:encase4} is our main technical tool for proving \cref{thm:stringmainfinite}.
We still need to introduce the notion of fortifications in order to give the precise technical statement that will allow for the inductive proof using \cref{lem:encase4}.

Let $G$ be a planar graph and consider some additional edge $e$ added to a face $F$.
Then $e$ splits $F$ into two faces $F_e$, $F_e'$.
We call these the \emph{sides} of $e$ in $G$.
We say that a planar graph $G'$ is a \emph{fortification} of a planar graph $G$ if $G'$ can be obtained from $G$ by adding edges to faces in such a way that for every $e\in E(G')\backslash E(G)$ we can pick a side $F_e$ of $e$ in $G$ so that for every distinct $e,e'\in E(G')\backslash E(G)$, the interiors of $F_e$ and $F_{e'}$ do not intersect, and no $F_e$ contains an edge incident to the outer face of $G$.
For an example of a fortification $G'$ of a planar graph $G$, see \cref{fig:fort}.

\begin{figure}
    \centering
    \includegraphics[width=0.45\linewidth]{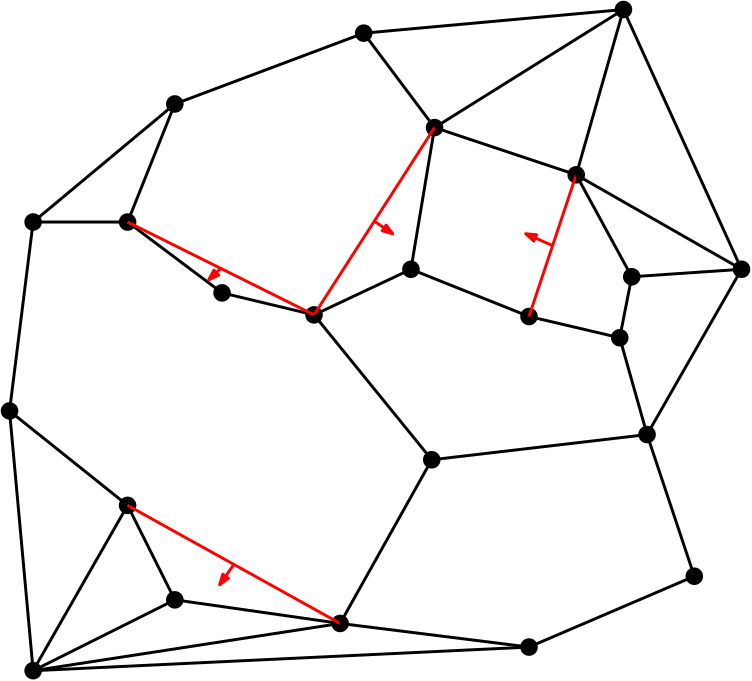}
    \caption{A fortification $G'$ of planar graph $G$ (which is the subgraph on the black edges.
    The arrows orthogonal to the additional red edges indicate their sides.}
    \label{fig:fort}
\end{figure}

Let $G$ be a planar graph and $\mathcal{H}$ a spanning collection of connected vertex sets $\mathcal{H}$ of $G$.
Let $\mathcal{H}_0$ be the collection of connected sets $H\in \mathcal{H}$ of $G$ that contain a vertex incident to the outer face of $G$.
we say that $(G',\mathcal{H'})$ is a \emph{$k$-fortification} of $(G,\mathcal{H})$ if 
\begin{itemize}
    \item $G'$ is a fortification of $G$,
    \item $\mathcal{H}'$ spans $G'$,
    \item $\mathcal{H}'$ can be obtained from $\mathcal{H}$ by adding connected sets $H'$ of $G$ such that each $I_{\mathcal{H}}(H')$ has weak diameter at most $k$ in $\textbf{RIG}(G,\mathcal{H}\backslash \mathcal{H}_0)$, and
    \item for every $uv \in E(G') \backslash E(G)$, there exists $H_u, H_v\in \mathcal{H}$ with $u\in H_u$ and $v\in H_v$, such that $d_{\textbf{RIG}(G,\mathcal{H} \backslash \mathcal{H}_0)}(H_u, H_v)\le k$.
\end{itemize}

We use the following technical statement to facilitate an inductive argument making use of \cref{lem:encase4}.

\begin{theorem}\label{thm:stringinduct}
    Let $G$ be a finite planar graph and $\mathcal{H}$ a spanning collection of connected vertex sets of $G$.
    Then there is a $(9242,80)$-impression $(G', \mathcal{M})$ of a $8$-fortification $(G',\mathcal{H}')$ of $(G,\mathcal{H})$ such that for every $M\in \mathcal{M}$ containing a vertex incident to the outer face of $G$, we have that $I_\mathcal{H'}(M)$ has weak diameter at most $9240$ in $\textbf{RIG}(G',\mathcal{H}')$, and for every $H\in \mathcal{H}$ containing a vertex incident to the outer face of $G$, we have that $I_\mathcal{M}(H)$ has weak diameter at most $9$ in $\textbf{IM}(G',\mathcal{M})$.
\end{theorem}

\begin{proof}
    We shall argue inductively on $|V(G)|$ (trivially the theorem holds for small $|V(G)|$).
    Let $\mathcal{H}_0$ be the collection of connected sets $H\in \mathcal{H}$ that contain a vertex incident to the outer face of $G$.
    By \cref{lem:encase4}, there is a $(9240,9,4,7)$-encasing $\mathcal{B}$ of $(G, \mathcal{H})$.

    Let $\mathcal{F}$ be the facial vertex sets of $G(\mathcal{B})$.
    Consider some $F\in \mathcal{F}$ and $H\in \mathcal{H}$ such that $H$ intersects both $\bigcup_{B\in \mathcal{B}}B$ and $F$.
    Let $R_{F,H}'$ be the vertices of $F$ that are encroached by $H$, and let $R_{F,H}=R_{F,H}' \cup (F\cap H)$.
    Consider some $H'\in \mathcal{H}$ intersecting $R_{F,H}$ but not $H$.
    Then, $H'$ intersects $R_{F,H}'$.
    Since $\mathcal{B}$ is a $4$-surround of $(G,\mathcal{H})$, the distance between $H'$ and $H$ in $\textbf{RIG}(G,\mathcal{H} \backslash \mathcal{H}_0)$ is at most 4.
    Therefore, $I_{\mathcal{H}}(R_{F,H})$ has weak diameter at most $2\cdot 4 =8$ in $\textbf{RIG}(G,\mathcal{H}\backslash \mathcal{H}_0)$.
    We wish to add edges to $G$ so that $R_{F,H}$ is a connected set.

    Consider some $v\in R_{F,H}$.
    By definition of being encroached, there exists connected set $U\subseteq \bigcup_{B\in \mathcal{B}}B$ of $G$ and edges $u_1v_1, \ldots, u_rv_r$ between $F$ and $\bigcup_{B\in \mathcal{B}}B$ that appear consecutively along the boundary of the face of $G[\bigcup_{B\in \mathcal{B}}B]$ corresponding to $F$, such that
\begin{itemize}
    \item $U$ is bounded by $H$ in $G[\bigcup_{B\in \mathcal{B}}B]$,
    \item for each $1\le i <r$, there is a facial walk of $G$ between $u_i$ and $u_{i+1}$ that is contained in $U$,
    \item $u_1,v_1,u_r,v_r\in H$, and
    \item $v\in \{v_1, \ldots , v_{r}\}$.
\end{itemize}
    Clearly $\{v_1, \ldots , v_{r}\}\subseteq R_{F,H}$.
    Let $E_{F,H}$ be the collection of such edges $v_iv_{i+1}$ (across each $v\in R_{F,H}$) that are not already an edge of $G$.
    For each such $v_iv_{i+1}\in E_{F,H}$, let $F_{v_iv_{i+1}}$ be the side of $v_iv_{i+1}$ drawn in the corresponding face of $G[\bigcup_{B\in \mathcal{B}}B]$ that contains the facial walk of $G$ between $u_i$ and $u_{i+1}$ that is contained in $U$ (note that the vertices of this facial walk are bounded $H$ in $G[\bigcup_{B\in \mathcal{B}}B]$ and the only two other vertices incident to this face are $v_i,v_{i+1}$).
    Since $\mathcal{B}$ is a 4-surround of $(G,\mathcal{H})$, it follows that for every $H'\in \mathcal{H}$ that contains a vertex incident to $F_{v_iv_{i+1}}$, the distance between $H$ and $H'$ in $\textbf{RIG}(G,\mathcal{H} \backslash \mathcal{H}_0)$ is at most $4$.
    This in particular implies that $F_{v_iv_{i+1}}$ will satisfy the 4th bullet in the definition of a $8$-fortification of $(G,\mathcal{H})$.
    This in particular implies that no vertex incident to $F_{v_iv_{i+1}}$ is incident to the outerface of $G$.
    Then observe that $R_{F,H}$ is a connected set in $G\cup E_{F,H}$.
    For an illustration of such a $R_{F,H}, E_{F,H}$, see \cref{fig:R_FH}.
    Let $G^*$ be the planar graph obtained from adding these edges from the union of all such $E_{F,H}$.
    Let $\mathcal{R}$ be the collection of all such $R_{F,H}$. Then every element of $\mathcal{R}$ is a connected set of $G^*$.

\begin{figure}
    \centering
    \includegraphics[width=0.7\linewidth]{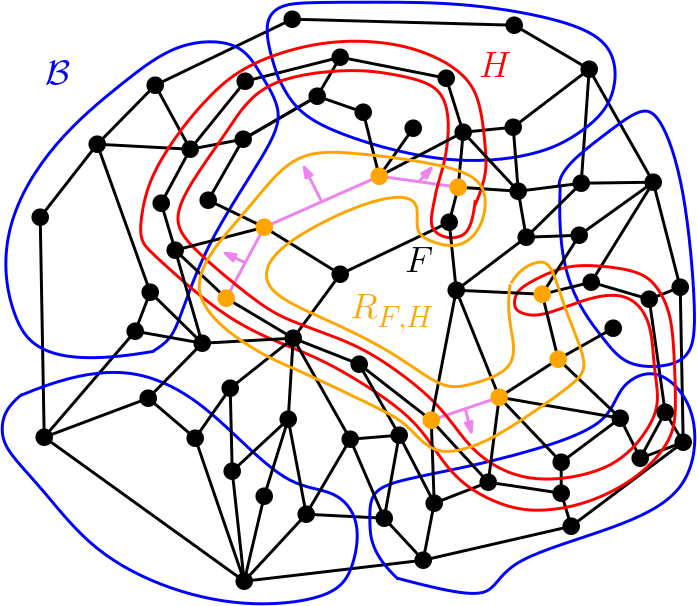}
    \caption{The (orange) vertex set $R_{F,H}$ given the (blue) encasing $\mathcal{B}$, the (red) $H\in \mathcal{H}$, and $F\in \mathcal{F}$.
    The orange vertices are the vertices of $R_{F,H}'$.
    The pink edges are $E_{F,H}$, and the arrows orthogonal to them indicate their sides.}
    \label{fig:R_FH}
\end{figure}

    Let $\mathcal{H}_C$ be the elements $H$ of $\mathcal{H}$ that are contained in $\bigcup_{B\in \mathcal{B}}B$, and
    let $\mathcal{H}_I$ be the elements $H$ of $\mathcal{H}$ that intersect but are not contained in $\bigcup_{B\in \mathcal{B}}B$.
    Let $\mathcal{R}^*= (\mathcal{H}\backslash (\mathcal{H}_C \cup \mathcal{H}_I)) \cup \mathcal{R}$.
    Let $G_0^*=G^*\backslash \left( \bigcup_{B\in \mathcal{B}} B \right)$, and let $\mathcal{R}_0^*$ be the elements of $\mathcal{R}^*$ that contain a vertex incident to the outer face of $G_0$.
    Note that $\mathcal{R}^*$ is a spanning collection of connected vertex sets of $G_0$ and that $\mathcal{R}\subseteq \mathcal{R}^*_0$.

    By the inductive hypothesis, there is a $(9242,80)$-impression $(G_0', \mathcal{M}')$ of a $8$-fortification $(G_0',\mathcal{R}')$ of $(G_0,\mathcal{R}^*)$ such that 
    for every $M\in \mathcal{M}'$ containing a vertex incident to the outer face of $G_0'$, we have that $I_{\mathcal{R}'}(M)$ has weak diameter at most $9240$ in $\textbf{RIG}(G_0',\mathcal{R}')$,
    for every $H\in \mathcal{R}_0^*$, we have that $I_{\mathcal{M}'}(H)$ has weak diameter at most $9$ in $\textbf{IM}(G_0', \mathcal{M'})$.
    Let $G'$ be obtained from $G$ by adding the edges in $E(G_0')\backslash E(G)$.
    Then, $G'$ is a fortification of $G$.
    Let $\mathcal{H}' = \mathcal{H} \cup \mathcal{R}'$.
    Then, $(G',\mathcal{H}')$ is furthermore a $8$-fortification of $(G, \mathcal{H})$.
    Let $\mathcal{M}= \mathcal{B} \cup  \mathcal{M}'$.
    It remains to show that $(G',\mathcal{M})$ provides the desired impression of $(G',\mathcal{H}')$.

    If $M\in \mathcal{M}$ contains a vertex incident to the outer face of $G'$, then $M\in \mathcal{B} = \mathcal{M} \backslash \mathcal{M}'$.
    In particular, $I_{\mathcal{H}}(M)=I_{\mathcal{H}'}(M)$ has weak diameter at most 9240 in $\textbf{RIG}(G,\mathcal{H})$, and therefore also in $\textbf{RIG}(G',\mathcal{H}')$.
    Consider some $H\in \mathcal{H}'$ that is contained in $\mathcal{H}_C$ (which in particular covers the case that $H$ contains a vertex incident to the outer face of $G'$).
    Then, $H\subseteq \bigcup_{B\in \mathcal{B}}$, so since $\mathcal{B}$ is a $(9240,9,4,7)$-encasing of $G$, $I_{\mathcal{B}}(H)$ has weak diameter at most $9$ in $\textbf{IM}(G,\mathcal{B})$. Therefore $I_{\mathcal{M}'}(H)$ has weak diameter at most $9$ in $\textbf{IM}(G,\mathcal{M})$.
    It remains to show that $(G',\mathcal{M})$ is a $(9242,80)$-impression of $(G',\mathcal{H}')$.

    If $M\in \mathcal{M}'\subseteq \mathcal{M}$ does not contain a vertex incident to the outer face of $G_0'$, then $I_{\mathcal{H}'}(M)=I_{\mathcal{R}'}(M)$ has weak diameter at most 9242 in $\textbf{RIG}(G_0',\mathcal{R}')$, and therefore also in $\textbf{RIG}(G',\mathcal{H}')$.
    If $M\in \mathcal{M}'\subseteq \mathcal{M}$ contains a vertex incident to the outer face of $G_0'$, then $I_{\mathcal{R}'}(M)$ has weak diameter at most 9240 in $\textbf{RIG}(G_0',\mathcal{R}')$.
    Since $I_{\mathcal{H}'}(M) \subseteq I_{\mathcal{R}'}(M) \cup  N_{\textbf{RIG}(G',\mathcal{H}')} ( I_{\mathcal{R}'}(M))$, it follows that $I_{\mathcal{H}'}(M)$ has weak diameter at most $9240+2 = 9242$ in $\textbf{RIG}(G',\mathcal{H}')$.
    So, $(G',\mathcal{M})$ is a $(9242,\infty)$-impression of $(G',\mathcal{H}')$.

    If $H\in \mathcal{R}' \subseteq \mathcal{H}'$, then $I_{\mathcal{M}'}(H)=I_{\mathcal{M}}(H)$ has weak diameter at most 80 in $\textbf{IM}(G_0', \mathcal{M}')$ and therefore also in $\textbf{IM}(G', \mathcal{M})$.
    To show that $(G',\mathcal{M}')$ is a $(9242,80)$-impression of $(G',\mathcal{H}')$, it now just remains to show that if $H\in \mathcal{H}_I$, then $I_\mathcal{M}(H)$ has weak diameter at most $80$ in $\textbf{IM}(G',\mathcal{M})$.

    Fix some $H\in \mathcal{H}_I$.
    Consider some $u,v\in H$ and let $M_u,M_v\in \mathcal{M}$ be such that $u\in M_u$ and $v\in M_v$.
    Since $\mathcal{B}$ is a 7-cage of $(G,\mathcal{H})$, there exists some sequence of vertex sets $A_0,  \ldots , A_t$ with $t\le 7$ such that
\begin{itemize}
    \item each $A_i$ is either a single element of $\mathcal{B}$ or a single element of $\mathcal{F}$ or the union of a subcollection $\mathcal{B}'\subseteq \mathcal{B}$ such that some connected set $H^*\subseteq H$ of $G$ (or equivalently $G'$) intersects each member of $\mathcal{B}'$ and is contained in $G[\bigcup_{B\in \mathcal{B}'}B]=G'[\bigcup_{B\in \mathcal{B}'}B]$,
    \item $u\in A_0$, 
    \item $v\in A_t$, and
    \item for every $1\le i <t$, we have that either $H$ intersects both $A_{i-1}$ and $A_i$, or neither $A_{i-1}$ or $A_i$ is an element of $\mathcal{F}$ and they touch.
\end{itemize}
If $A_i\in \mathcal{B}\subseteq \mathcal{M}$, then clearly $\{A_i\}$ has weak diameter 0 in $\textbf{IM}(G',\mathcal{M})$.
If $A_i$ is the union of a subcollection $\mathcal{B}'\subseteq \mathcal{B}$ such that some connected set $H^*\subseteq H$ of $G$ intersects each member of $\mathcal{B}'$ and is contained in $G[\bigcup_{B\in \mathcal{B}'}B]$, then it follows from $\mathcal{B}$ being a $(9240,9,4,7)$-encasing of $(G, \mathcal{H})$ that $A_i$ has weak diameter at most $9$ in $\textbf{IM}(G,\mathcal{B})$ and therefore in $\textbf{IM}(G',\mathcal{M})$.
If $A_i$ is a single element $F$ of $\mathcal{F}$, then $R_{F,H}$ contains $H\cap F$.
As $R_{F,H} \in \mathcal{R}_0^*$, we have that $I_{\mathcal{M'}}(R_{F,H})$ has weak diameter at most 9 in $\textbf{IM}(G_0,\mathcal{M}')$, and therefore $I_{\mathcal{M}}(R_{F,H})$ has weak diameter at most 9 in $\textbf{IM}(G',\mathcal{M})$.
As $t\le 7$, it now follows that $d_{\textbf{IM}(G',\mathcal{M})}(M_u,M_v) \le 7+ 8\cdot 9  =80$.
Therefore, $I_\mathcal{M}(H)$ has weak diameter at most $80$ in $\textbf{IM}(G',\mathcal{M})$.

Hence $(G',\mathcal{M})$ is a $(9242,80)$-impression of $(G',\mathcal{H}')$, as desired.
\end{proof}

By \cref{thm:stringinduct} and \cref{lem:impression2}, we obtain the following.

\begin{theorem}\label{thm:stringimpression}
    Let $G$ be a finite planar graph and $\mathcal{H}$ a spanning collection of connected vertex sets of $G$.
    Then there is a $(73936,80)$-impression $(G', \mathcal{M})$ of $(G',\mathcal{H})$, where $G'$ is a fortification of $G$ and $\mathcal{H}$ $8$-spans $G'$.
\end{theorem}

Since for every finite string graph $S$ there is a finite planar graph $G$ and a spanning collection $\mathcal{H}$ of connected vertex sets of $G$ such that $\textbf{RIG}(G,\mathcal{H}) = S$ \cite{lee2016separators}, by \cref{thm:stringimpression} and \cref{prop:impression}, we finally get that string graphs are quasi-isometric to planar graphs.

\begin{theorem}\label{thm:stringquasi}
    Let $S$ be a finite string graph. Then there is a planar graph $G$ and a function $f:V(S) \to V(G)$ such that for all $u,v\in V(S)$,
    \[
    \frac{1}{73944} d_S(u,v) - \frac{73936}{73944}  \le d_G(f(u),f(v))
    \le 160 d_S(u,v), 
    \]
    and for every $v\in V(G)$ there exists some $u\in V(S)$ such that $d_G(f(u),v)\le 80$.
\end{theorem}

\cref{thm:stringmainfinite} (and thus \cref{thm:stringmain}) then finally follows from \cref{thm:stringquasi} and \cref{lem:quasibi}.

\section{Riemannian surfaces}\label{sec:Riemannian}

In this section, we prove \cref{thm:RiemannianTriangulation}.

In addition to \cref{thm:stringimpression}, our main tool is the following recent theorem of Georgakopoulos and Vigolo \cite{georgakopoulos2026triangulating}.
(We delay giving certain definitions on Riemannian surfaces until after the proof of our last required graph theoretic lemma (which is \cref{lem:metricEuler}).)

\begin{theorem}[Georgakopoulos, Vigolo \cite{georgakopoulos2026triangulating}]\label{thm:triangulate}
    Let $\Sigma$ be a complete Riemannian surface.
    Then $\Sigma$ admits a piecewise smooth triangulation $H\subset \Sigma$ by a locally finite $(0,1]$-metric graph $H$ such that $H \hookrightarrow \Sigma$ is a quasi-isometry.
\end{theorem}

Roughly speaking, \cref{thm:triangulate} allows us to just deal with modifying some triangulation by a metric graph to obtain a quasi-isometric triangulation by a simplicial graph as in \cref{thm:RiemannianTriangulation}.

We begin with considering metric graphs in the plane.
A graph is an internally triangulated planar graph if it has an embedding in the plane in which every face except possibly the outerface is a triangle.
Observe that in the case of internally triangulated planar graphs, \cref{thm:stringimpression} simplifies to directly finding an impression of $(G,\mathcal{H})$ since in this case $G'=G$.
We require a version more directly applicable to metric graphs.
We obtain this from \cref{thm:stringimpression} with essentially the same trick as in \cref{Lem:metrictostring} and using \cref{prop:impression}.

\begin{lemma}\label{lem:metricinternaltriangulate}
    Let $H$ be a finite $(0,1]$-metric internally triangulated planar graph.
    Then there exists a partition of $V(H)$ into connected vertex sets $\mathcal{M}$ such that for every $u,v\in V(H)$, we have that 
    \[
    \frac{1}{225000}d_H(u,v) - 1
    \le
    d_{\textbf{IM}(H,\mathcal{M})}(f(u),f(v))
    \le
    160d_H(u,v) + 160,
    \]
    where $f(u)=M$ for every $u\in M \in \mathcal{M}$.
\end{lemma}

\begin{proof}
    Similarly to in the proof of \cref{Lem:metrictostring}, for each $u\in V(H)$, we let $S_u$ be the set of vertices of $H$ at distance at most $\frac{3}{2}$ from $u$.
    Let $\mathcal{S}= \{S_u : u\in V(H)\}$.
    Clearly $\mathcal{S}$ is a spanning collection of connected vertex sets of $H$.
    
    Consider some $u,v\in V(H)$ and let $P$ be a geodesic path from $u$ to $v$ in $H$.
    Choose vertices $x_0,\ldots ,x_t$ of $H$ along $P$ so that $x_0=u$, $x_t=v$ and for each $0\le i <t$, the distance along $P$ between $x_i$ and $x_{i+1}$ is in $[1,2]$ (or if $d_H(u,v) < 1$, we just take $x_0=u$, $x_1=v$).
    So, $\lceil \frac{1}{2}d_H(u,v) \rceil \le  t\le \lfloor d_H(u,v) \rfloor$.
    For each $0\le i <t$, we have that $S_{x_i}$ intersects $S_{x_{i+1}}$ since $H$ is a $(0,1]$-metric graph and $d_H(x_i, x_{i+1}) + 1 \le 3 = 2 \cdot \frac{3}{2}$.
    It therefore follows that for every pair $u,v\in V(H)$, we have that
    \[
    \frac{1}{3}d_H(u,v)
    \le
    d_{\textbf{RIG}(H,\mathcal{S})}(S_u,S_v)
    \le
    d_H(u,v) + 1.
    \]

    By \cref{thm:stringimpression}, there is a $(73936,80)$-impression $(H, \mathcal{M})$ of $(H,\mathcal{S})$.
    By \cref{prop:impression}, it then follows that for every $u,v\in V(H)$, we have that
    \[
    \frac{1}{73944} d_{\textbf{RIG}(H,\mathcal{S})}(S_u,S_v) - \frac{73936}{73944}  
    \le 
    d_{\textbf{IM}(H,\mathcal{M})}(f(u),f(v))
    \le 
    160 d_{\textbf{RIG}(H,\mathcal{S})}(S_u,S_v). 
    \]

    Therefore, for every $u,v\in V(H)$, we have that
    \[
    \frac{1}{225000}d_H(u,v) - 1
    \le
    d_{\textbf{IM}(H,\mathcal{M})}(f(u),f(v))
    \le
    160d_H(u,v) + 160,
    \]
    as desired.
\end{proof}

Next we extend \cref{lem:metricinternaltriangulate} to triangulations of higher genus surfaces.
A graph triangulates a surface without boundary if it has an embedding in which every face is a non-degenerate triangle (meaning it is bounded by three distinct vertices and edges) and is homeomorphic to a disk.
In extending \cref{lem:metricinternaltriangulate}, we just get a some extra additive distortion which is linear in the Euler genus of the surface.

\begin{lemma}\label{lem:metricEulerfinite}
    Let $H$ be a finite $(0,1]$-metric graph that triangulates a surface $\Sigma$ without boundary and with Euler genus $g$.
    Then there exists a partition of $V(H)$ into connected vertex sets $\mathcal{M}$ such that for every $u,v\in V(H)$, we have that 
    \[
    \frac{1}{225000}d_H(u,v) - 1
    \le
    d_{\textbf{IM}(H,\mathcal{M})}(f(u),f(v))
    \le
    160d_H(u,v) + 160 +1280g,
    \]
    where $f(u)=M$ for every $u\in M \in \mathcal{M}$.
\end{lemma}

\begin{proof}
    Choose a vertex $r$ of $H$ and construct a spanning tree $T$ of $H$ with root $r$ by ordering the vertices of $H$ according to their distance from $r$ and adding them and some incident edge greedily to the tree so that for each $x\in V(H)$, we have that $d_T(r,x)=d_H(r,x)$.
    
    In the dual of the embedding of $H'$ in $\Sigma$ (defined as the graph $H^*$ with a vertex in each face, the same edge set as $H$ and with each edge $e$ in $H^*$ being incident to the two vertices corresponding to the two faces that $e$ is incident to in $H$, (possibly allowing loops or multiple edges)), choose some spanning tree $T^*$ with $E(T)\cap E(T^*) = \emptyset$.
    Let $F$ be the edges of $H$ not in $T$ or $T^*$.
    By Euler's formula, we have that $|F|=g$.
    Let $T'$ be the minimal subtree of $T$ containing the ends of the edges of $F$.
    Note that $T'$ is the union of at most $2g$ geodesics of $H$.
    Now observe that $G=H - V(T')$ is a $(0,1[$-metric internally triangulated planar graph.

    By \cref{lem:metricinternaltriangulate}, there is a partition of $V(G)$ into connected vertex sets $\mathcal{X}$ such that  for every $u,v\in V(G)$, we have that 
    \[
    \frac{1}{225000}d_G(u,v) -1
    \le
    d_{\textbf{IM}(G,\mathcal{X})}(f_{\mathcal{X}}(u),f_{\mathcal{X}}(v))
    \le
    160d_G(u,v) + 160,
    \]
    where $f_\mathcal{X}(u)=X$ for all $u\in X \in \mathcal{X}$.

    Let $\mathcal{Y}_1=\{N_{T'}^{3}[r]\}$ and for each positive integer $i\ge 2$, let $\mathcal{Y}_i$ be the collection of maximal connected vertex sets in $T[N_{T'}^{3i}[r]\backslash N_{T'}^{3(i-1)}[r]]$.
    Let $\mathcal{Y} = \bigcup_{i=1}^\infty \mathcal{Y}_i$.
    Since $T$ is a $(0,1]$-metric tree, observe that for every $u,v\in V(T')$, we have that
    \[
    \frac{1}{6}d_{T'}(u,v) -1
    \le
    d_{\textbf{IM}({T'},\mathcal{Y})}(f_{\mathcal{Y}}(u),f_{\mathcal{Y}}(v))
    \le
    d_{T'}(u,v) + 2,
    \]
    where $f_\mathcal{Y}(u)=Y$ for all $u\in Y \in \mathcal{Y}$.

    Let $\mathcal{M} = \mathcal{X} \cup \mathcal{Y}$, and let $f(u)=M$ for every $u \in M \in \mathcal{M}$.
    Clearly $\mathcal{M}$ is a partition of $V(H)$ into connected vertex sets.
    Each $X\in \mathcal{X}$ has weak diameter at most 225000 in $G$ and therefore also in $H$.
    Similarly, each $Y\in \mathcal{Y}$ has weak diameter at most $6<225000$ in $T'$ and therefore also in $H$.
    So, each $M\in \mathcal{M}$ has weak diameter at most 225000 in $H$, it follows that for every $u,v\in V(G)$, we have that 
    $225000^{-1}d_H(u,v) - 1
    \le
    d_{\textbf{IM}(H,\mathcal{M})}(f(u),f(v))$.

    Consider some $u,v\in V(H)$.
    Since $T'$ is the union of at most $2g$ geodesic paths in $H$, we can find a path $P=P_1e_1P_2\cdots e_{n-1}P_n$ in $H$ between $u$ and $v$ such that
    \begin{itemize}
        \item $P$ has length $d_H(u,v)$,
        \item $n\le 4g+1$,
        \item for each $1\le i <n$, $e_i$ is an edge of $H$ not contained in $E(T')$ or $E(G)$, and
        \item for each $1\le i \le n$, $P_i$ is a geodesic of either $T'$ or $G$.
    \end{itemize}
    For each $1\le i \le n$, let $p_{2i-1}$ be the starting vertex of $P_i$, and let $p_{2i}$ be the end vertex of $P_i$ (these vertices need not be distinct).
    Clearly for each $1\le i < n$, we have $d_{\textbf{IM}(H',\mathcal{M})}(f(p_{2i}),f(p_{2i+1}))
    \le 1$.
    So, for each $1\le i < 2n$, we have that 
    $d_{\textbf{IM}(H',\mathcal{M})}(f(p_{2i-1}),f(p_{2i}))
    \le
    160d_H(p_{2i-1},p_{2i}) + 160$, since each $P_j$ is a geodesic in both $H$ and either $T'$ or $G$.
    Then,
    \begin{align*}
        d_{\textbf{IM}(H,\mathcal{M})}(f(u),f(v))
        & \le 
        \sum_{i=1}^{2i-1} d_{\textbf{IM}(H,\mathcal{M})}(f(p_i),f(p_{i+1}))
        \\
        & \le 
        \sum_{i=1}^{2i-1} \left( 160 d_H(p_i,p_{i+1}) + 160 \right)
        \\
        & = 
        160d_H(u,v)
        + 160(2i-1)
        \\
        & \le
        160 d_H(u,v)
        + 160(8g+1)
        \\
        & =
        160 d_H(u,v)
        + 160 + 1280g,
    \end{align*}
    as desired.
\end{proof}

For our applications to Riemannian surfaces, we further need the following version of \cref{lem:metricEulerfinite} for locally finite graphs, which is a standard extension.

\begin{lemma}\label{lem:metricEuler}
    Let $H$ be a locally finite $(0,1]$-metric graph that triangulates a surface $\Sigma$ without boundary and with Euler genus $g$.
    Then there exists a partition of $V(H)$ into connected vertex sets $\mathcal{M}$ such that for every $u,v\in V(H)$, we have that  
    \[
    \frac{1}{225000}d_H(u,v) - 1
    \le
    d_{\textbf{IM}(H,\mathcal{M})}(f(u),f(v))
    \le
    160d_H(u,v) + 160 +1280g,
    \]
    where $f(u)=M$ for every $u\in M \in \mathcal{M}$.
\end{lemma}

\begin{proof}
    Let $\mathcal{X}=\{X_k : k\ge 1\}$ be the collection of all finite subsets of $V(H)$.
    Fix a vertex $r$ of $H$.
    For each positive integer $i$, let $H_i$ be some finite $(0,1]$-metric graph that triangulates a surface $\Sigma'$ of Euler genus at most $g$, contains the vertices $N_H^{i}[r]$ and with $N_{H_i}^{i}[ r]=N_H^{i}[r]$ (note that $\Sigma'$ may need to be chosen to be some compactification of $\Sigma$).
    For each $i$, let $\mathcal{M}_i$ be a partition of $V(H_i)$ into connected vertex sets as given by \cref{lem:metricEulerfinite}.
    
    Let $J_0=\mathbb{N}$ and for each $k\ge 1$ in order, let $J_k\subseteq J_{k-1}$ be an infinite subset such that either $X_k\in \mathcal{M}_i$ for every $i\in J_k$, or $X_k \not\in \mathcal{M}_i$ for every $i\in J_k$.
    Let $\mathcal{M}$ be the subset of $\mathcal{X}$ consisting of all $X_k$ such that $X_k\in \mathcal{M}_i$ for every $i\in J_k$.
    
    For each pair of positive integers $i,t$ with $i\ge t+22500$ and $t\ge 22500$, let $\mathcal{Y}_{i,t}$ be the subset of $\mathcal{M}_i$ consisting of all $M\in \mathcal{M}_i$ intersecting $N_H^t[r]$.
    Note that for every positive integer $t$, there are only finitely many possibilities for such $\mathcal{Y}_{i,t}$ and also $N_H^{t-22500}[r] \subseteq  \bigcup_{Y\in \mathcal{Y}_{i,t}} Y$.
    So, it follows that $\mathcal{M}$ is a partition of $V(H)$ into connected sets.
    It is furthermore straightforward to check that $\mathcal{M}$ satisfies the conclusion of the lemma by using the fact that for every pair $i,t$ with $i\ge 10^6t$, we have for every $u,v\in N_H^t[r]$ that $d_{H}(u,v)=d_{H_i}(u,v)$ and $d_{\textbf{IM}(H,\mathcal{M}_i)}(f_i(u),f_i(v)) = d_{\textbf{IM}(H_i,\mathcal{M}_i)}(f_i(u),f_i(v))$, where $f_i(u)=M$ for every $u\in M\in \mathcal{M}_i$.
\end{proof}

With \cref{lem:metricEuler} we are almost ready to prove \cref{thm:RiemannianTriangulation}.
First we give some required preliminaries on Riemannian surfaces.

A \emph{Riemannian surface} is a surface $\Sigma$ together with a metric $d_\Sigma$ defined by a scalar product on the tangent space of every point.
A Riemannian surface $\Sigma$ is \emph{complete} if the metric space $(\Sigma, d_\Sigma)$ is complete.
Note that compact Riemannian surfaces are also complete.
By the Hopf–Rinow \cite{hopf1931ueber} theorem, in a complete Riemannian surface $\Sigma$, subsets that are bounded and closed are compact.
Restricting a complete Riemannian surface to a bounded compact subset homeomorphic to a disk gives a complete Riemannian surface homeomorphic to a disk.
For more on Riemannian surfaces, see \cite{spivak1979comprehensive}.

A \emph{cell decomposition} of a complete Riemannian surface $\Sigma$ (without boundary) can be described as a locally finite embedded graph $G \subset \Sigma$ such that every face $F$ is homeomorphic to the unit open ball $B \subset \mathbb{R}^2$ via a homeomorphism $B \to \Sigma$ that extends to a continuous map of the unit disc $D$ mapping its boundary into $G$.
This mapping defines a boundary path in $G$, and the boundary size of $F$ is the number of edges in its boundary path, counted with multiplicity.
If every face has boundary size exactly 3 and furthermore the boundary paths all consist of 3 distinct vertices and edges, then such a cell decomposition is a \emph{triangulation}.
A cell decomposition is \emph{piecewise smooth} if every edge in the embedding is piecewise smooth.

We need the following technical lemma of Georgakopoulos and Vigolo \cite[Lemma 2.8]{georgakopoulos2026triangulating} on modifying embedded graphs in Riemannian surfaces (actually this is a slightly less general version).

\begin{lemma}[Georgakopoulos and Vigolo \cite{georgakopoulos2026triangulating}]\label{lem:RieTree}
    Let $\Sigma$ be a complete Riemannian surface homeomorphic to a disk with boundary $\partial\Sigma$, and let $T \subset \Sigma$ be an embedded finite rooted tree consisting of piecewise smooth curves of finite length with root $o$ and leaves (distinct from $o$ if $o$ is a leaf) $\{p_1 , \ldots , p_n\}$ such that $T\cap \partial\Sigma = \{p_1 , \ldots , p_n\}$.
    
Then for every $\epsilon > 0$ we may find piecewise smooth curves $\gamma_i$ joining $p_i$ to $o$ such that:
\begin{itemize}
    \item different $\gamma_i$ only meet at $o$;
    \item $\bigcup_{i=1}^n \gamma_i$ is within Hausdorff distance $\epsilon$ from $T$;
    \item $\left(\bigcup_{i=1}^n \gamma_i\right)\cap \partial\Sigma = \{p_1 , \ldots , p_n\}$.
\end{itemize}
\end{lemma}

Now, we prove the following slight weakening of \cref{thm:RiemannianTriangulation}.
With \cref{lem:metricEuler} in hand, we roughly follow the idea of the proof of \cite[Theorem 4.1]{georgakopoulos2026triangulating}, applying it along with \cref{thm:triangulate} and \cref{lem:RieTree}.

\begin{theorem}\label{thm:RiemannianTriangulationB}
    Let $\Sigma$ be a complete Riemannian surface (without boundary) and with bounded Euler genus.
    Then there is a cell decomposition $G\subset \Sigma$ where every face has boundary of size at most three and $G^{(1)} \hookrightarrow \Sigma$ is a quasi-isometry.
\end{theorem}

\begin{proof}
    Let $H\subset \Sigma$ be a piecewise smooth triangulation where $H$ is a locally finite $(0,1]$-metric graph $H$ such that $H\hookrightarrow \Sigma$ is a quasi-isometry, as given by \cref{thm:triangulate}.
    Since $H$ triangulates a surface of bounded genus, by \cref{lem:metricEuler}, there exists some $X\subseteq V(H)$ and a partition $\mathcal{M}=\{M_x : x\in X\}$ of $V(H)$ into connected sets with $x\in M_x$ for every $x\in X$ and such that the map $f:V(H)\to \mathcal{M}$ with $f(y)=x$ for every $y\in M_x$ is a quasi-isometry from $H$ to $\textbf{IM}(H,\mathcal{M})$.
    For each $x\in X$, let $T_x$ be a spanning tree of $H[M_x]$ with root $x$.
    For some sufficiently small $\epsilon_x>0$ we can choose some $N_x\subseteq \Sigma$ with $N_x=\{p : d_\Sigma(T_x,p) \le \epsilon_x\}$ such that $N_x\cap H$ is a tree $T_x^*$ (whose leaves were not vertices of $H$) and $N_x$ is homeomorphic to a disk.
    Restricting $\Sigma$ to $N_x$ yields a complete Riemannian surface homeomorphic to the disk.
    The leaves of $T_x^*$ are exactly the points of $T_x^*$ intersecting $\partial N_x$.
    Applying \cref{lem:RieTree} to every $T_x^*\subseteq N_x$ yields a cell decomposition $G\subset \Sigma$ where every face has boundary size at most three.
    For an illustration of this, see \cref{fig:rieman}.
    Since the graph $G$ (or essentially its simplicial 1-skeleton $G^{(1)}$) is obtained from $\textbf{IM}(H,\mathcal{M})$ by possibly adding loops and multiple edges, it can be seen from the quasi-isometry $f$ that $G^{(1)} \hookrightarrow \Sigma$ is also a quasi-isometry.
\end{proof}

\begin{figure}
    \centering
    \includegraphics[width=0.9\linewidth]{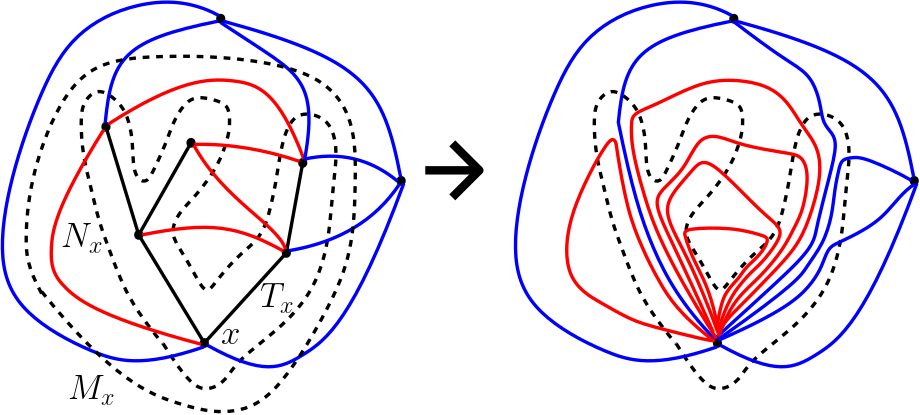}
    \caption{On the left is an illustration of the triangulation $H\subset \Sigma$. The induced subgraph $H[M_x]$ is the one consisting of the (black) rooted tree $T_x$ and the other (red) edges between vertices of $M_x$. The rest of $H$ is blue.
    For each such $x$, we apply \cref{lem:RieTree} to modify the cell decomposition as on the right.
    By doing this for every $x\in X$, we obtain the desired cell decomposition $G\subset \Sigma$.}
    \label{fig:rieman}
\end{figure}

Georgakopoulos and Vigolo \cite[Lemma 2.7]{georgakopoulos2026triangulating} showed that by taking some subgraph $G_1$ of $G$, such a cell decomposition can be refined so that every face has boundary of size exactly 3, while preserving the quasi-isometry $G_1^{(1)} \hookrightarrow \Sigma$.
They also further observed \cite[Remark 2.6]{georgakopoulos2026triangulating} that by taking the barycentric
subdivision of $G_1\subset \Sigma$, such a cell decomposition can be turned into a triangulation $G_2 \subset \Sigma$, again while preserving the quasi-isometry $G_2^{(1)} \hookrightarrow \Sigma$.
Thus, we obtain \cref{thm:RiemannianTriangulation} from \cref{thm:RiemannianTriangulationB}.

It can easily be checked that the proof of \cref{thm:RiemannianTriangulation} yields a $(M,A+O(g))$-quasi-isometry where $M$ and $A$ are absolute constants independent of the Euler genus.
In fact, a close inspection shows that it yields a $(10^6, 10^6+10^4g)$-quasi-isometry.

\section*{Acknowledgements}

This work was initiated at the Workshop on Graph Product Structure Theory (BIRS21w5235) at the Banff International Research Station, 21–26 November
2021.
We are grateful to the organisers and
participants for providing a stimulating research environment.
We thank a number of people for helpful discussions and encouragement on the project including
Sandra Albrechtsen,
Marthe Bonamy,
Linda Cook,
Louis Esperet,
Agelos Georgakopoulos, 
Robert Hickingbotham,
Rose McCarty,
Harry Petyt,
Federico Vigolo, and
Alexandra Wesolek.

\bibliographystyle{amsplain}
{\small
\bibliography{QI}}

@article{bonamy2023asymptotic,
  title={Asymptotic dimension of minor-closed families and Assouad--Nagata dimension of surfaces},
  author={Bonamy, Marthe and Bousquet, Nicolas and Esperet, Louis and Groenland, Carla and Liu, Chun-Hung and Pirot, Fran{\c{c}}ois and Scott, Alexander},
  journal={Journal of the European Mathematical Society},
  volume={26},
  number={10},
  pages={3739--3791},
  year={2023}
}

@article{georgakopoulos2023graph,
  title={Graph minors and metric spaces},
  author={Georgakopoulos, Agelos and Papasoglu, Panos},
  journal={arXiv preprint arXiv:2305.07456},
  year={2023}
}

@article{distel2023proper,
  title={Proper minor-closed classes of graphs have {A}ssouad-{N}agata dimension 2},
  author={Distel, Marc},
  journal={arXiv preprint arXiv:2308.10377},
  year={2023}
}

@article{liu2023assouad,
  title={Assouad--{N}agata dimension of minor-closed metrics},
  author={Liu, Chun-Hung},
  journal={Proceedings of the London Mathematical Society},
  volume={130},
  number={3},
  pages={e70032},
  year={2025},
  publisher={Wiley Online Library}
}

@article{fujiwara2023coarse,
  title={A coarse-geometry characterization of cacti},
  author={Fujiwara, Koji and Papasoglu, Panos},
  journal={arXiv preprint arXiv:2305.08512},
  year={2023}
}

@article{fujiwara2021asymptotic,
  title={Asymptotic dimension of planes and planar graphs},
  author={Fujiwara, Koji and Papasoglu, Panos},
  journal={Transactions of the American Mathematical Society},
  volume={374},
  number={12},
  pages={8887--8901},
  year={2021}
}

@article{davies2024fat,
  title={Fat minors cannot be thinned (by quasi-isometries)},
  author={Davies, James and Hickingbotham, Robert and Illingworth, Freddie and McCarty, Rose},
  journal={arXiv preprint arXiv:2405.09383},
  year={2024}
}

@article{davies2025coarse,
  title={Coarse structure for minor-free metrics},
  author={Davies, James and Distel, Marc and Hickingbotham, Robert},
  journal={(In preparation)}
}

@article{davies2025strong,
  title={Strongly sublinear separators and bounded asymptotic dimension for sphere intersection graphs},
  author={Davies, James and Georgakopoulos, Agelos and Hatzel, Meike and McCarty, Rose},
  journal={arXiv preprint arXiv:2504.00932},
  year={2025}
}

@article{macmanus2024note,
  title={A note on quasi-transitive graphs quasi-isometric to planar ({C}ayley) graphs},
  author={MacManus, Joseph},
  journal={arXiv preprint arXiv:2407.13375},
  year={2024}
}

@article{nguyen2025counterexample,
  title={A counterexample to the coarse Menger conjecture},
  author={Nguyen, Tung and Scott, Alex and Seymour, Paul},
  journal={Journal of Combinatorial Theory, Series B},
  volume={173},
  pages={68--82},
  year={2025},
  publisher={Elsevier}
}

@article{albrechtsen2024menger,
  title={A Menger-type theorem for two induced paths},
  author={Albrechtsen, Sandra and Huynh, Tony and Jacobs, Raphael W. and Knappe, Paul and Wollan, Paul},
  journal={SIAM Journal on Discrete Mathematics},
  volume={38},
  number={2},
  pages={1438--1450},
  year={2024},
  publisher={SIAM}
}

@article{macmanus2023accessibility,
  title={Accessibility, planar graphs, and quasi-isometries},
  author={MacManus, Joseph},
  journal={arXiv preprint arXiv:2310.15242},
  year={2023}
}

@article{albrechtsen2024characterisation,
  title={A characterisation of graphs quasi-isometric to ${K_4}$-minor-free graphs},
  author={Albrechtsen, Sandra and Jacobs, Raphael W. and Knappe, Paul and Wollan, Paul},
  journal={arXiv preprint arXiv:2408.15335},
  year={2024}
}

@article{fox2012string,
  title={String graphs and incomparability graphs},
  author={Fox, Jacob and Pach, J{\'a}nos},
  journal={Advances in mathematics},
  volume={230},
  number={3},
  pages={1381--1401},
  year={2012},
  publisher={Academic Press USA}
}

@article{pach2002recognizing,
  title={Recognizing string graphs is decidable},
  author={Pach, J{\'a}nos and T{\'o}th, G{\'e}za},
  journal={Discrete \& Computational Geometry},
  volume={28},
  pages={593--606},
  year={2002},
  publisher={Springer}
}

@article{fox2010separator,
  title={A separator theorem for string graphs and its applications},
  author={Fox, Jacob and Pach, J{\'a}nos},
  journal={Combinatorics, Probability and Computing},
  volume={19},
  number={3},
  pages={371--390},
  year={2010},
  publisher={Cambridge University Press}
}

@article{fox2014applications,
  title={Applications of a new separator theorem for string graphs},
  author={Fox, Jacob and Pach, J{\'a}nos},
  journal={Combinatorics, Probability and Computing},
  volume={23},
  number={1},
  pages={66--74},
  year={2014},
  publisher={Cambridge University Press}
}

@article{fox2008separator,
  title={Separator theorems and {T}ur{\'a}n-type results for planar intersection graphs},
  author={Fox, Jacob and Pach, J{\'a}nos},
  journal={Advances in Mathematics},
  volume={219},
  number={3},
  pages={1070--1080},
  year={2008},
  publisher={Elsevier}
}

@article{lee2016separators,
  title={Separators in region intersection graphs},
  author={Lee, James R},
  journal={arXiv preprint arXiv:1608.01612},
  year={2016}
}

@article{tomon2023string,
  title={String graphs have the {E}rd{\H{o}}s--{H}ajnal property},
  author={Tomon, Istv{\'a}n},
  journal={Journal of the European Mathematical Society},
  volume={26},
  number={1},
  pages={275--287},
  year={2023}
}

@article{spivak1979comprehensive,
  title={A comprehensive introduction to differential geometry, Publish or Perish},
  author={Spivak, Michael},
  journal={Inc., Berkeley},
  volume={2},
  year={1979}
}

@article{albrechtsen2025counterexample,
  title={Counterexample to the conjectured coarse grid theorem},
  author={Albrechtsen, Sandra and Davies, James},
  journal={arXiv preprint arXiv:2508.15342},
  year={2025}
}

@article{chepoi2012constant,
  title={Constant approximation algorithms for embedding graph metrics into trees and outerplanar graphs},
  author={Chepoi, Victor and Dragan, Feodor F and Newman, Ilan and Rabinovich, Yuri and Vaxes, Yann},
  journal={Discrete \& Computational Geometry},
  volume={47},
  number={1},
  pages={187--214},
  year={2012},
  publisher={Springer}
}

@article{matouvsek2014near,
  title={Near-optimal separators in string graphs},
  author={Matou{\v{s}}ek, Ji{\v{r}}{\'\i}},
  journal={Combinatorics, Probability and Computing},
  volume={23},
  number={1},
  pages={135--139},
  year={2014},
  publisher={Cambridge University Press}
}

@article{ABRISHAMI2025string,
  title={Burling graphs in graphs with large chromatic number},
  author={Abrishami, Tara and Bria{\'n}ski, Marcin and Davies, James and Du, Xiying and Masa{\v{r}}{\'\i}kov{\'a}, Jana and Rz{\k{a}}{\.z}ewski, Pawe{\l} and Walczak, Bartosz},
  journal={arXiv preprint arXiv:2510.19650},
  year={2025}
}

@article{benzer1959topology,
  title={On the topology of the genetic fine structure},
  author={Benzer, Seymour},
  journal={Proceedings of the National Academy of Sciences},
  volume={45},
  number={11},
  pages={1607--1620},
  year={1959}
}

@article{sinden1966topology,
  title={Topology of thin film {R}{C} circuits},
  author={Sinden, Frank W},
  journal={Bell System Technical Journal},
  volume={45},
  number={9},
  pages={1639--1662},
  year={1966},
  publisher={Wiley Online Library}
}

@article{ehrlich1976intersection,
  title={Intersection graphs of curves in the plane},
  author={Ehrlich, Gideon and Even, Shimon and Tarjan, Robert Endre},
  journal={Journal of Combinatorial Theory, Series B},
  volume={21},
  number={1},
  pages={8--20},
  year={1976},
  publisher={Elsevier}
}

@article{kratochvil1991string,
  title={String graphs. {I}{I}. {R}ecognizing string graphs is {N}{P}-hard},
  author={Kratochv{\'\i}l, Jan},
  journal={Journal of Combinatorial Theory, Series B},
  volume={52},
  number={1},
  pages={67--78},
  year={1991},
  publisher={Academic Press}
}

@article{kratochvil1991stringexp,
  title={String graphs requiring exponential representations},
  author={Kratochv{\'\i}l, Jan and Matou{\v{s}}ek, Ji{\v{r}}{\'\i}},
  journal={Journal of Combinatorial Theory, Series B},
  volume={53},
  number={1},
  pages={1--4},
  year={1991},
  publisher={Elsevier}
}

@article{hopcroft1974efficient,
  title={Efficient planarity testing},
  author={Hopcroft, John and Tarjan, Robert},
  journal={Journal of the ACM (JACM)},
  volume={21},
  number={4},
  pages={549--568},
  year={1974},
  publisher={ACM New York, NY, USA}
}

@inproceedings{schaefer2002recognizing,
  title={Recognizing string graphs in {N}{P}},
  author={Schaefer, Marcus and Sedgwick, Eric and {\v{S}}tefankovi{\v{c}}, Daniel},
  booktitle={Proceedings of the thiry-fourth annual ACM symposium on Theory of computing},
  pages={1--6},
  year={2002}
}

@article{pawlik2014triangle,
  title={Triangle-free intersection graphs of line segments with large chromatic number},
  author={Pawlik, Arkadiusz and Kozik, Jakub and Krawczyk, Tomasz and Laso{\'n}, Micha{\l} and Micek, Piotr and Trotter, William T and Walczak, Bartosz},
  journal={Journal of Combinatorial Theory, Series B},
  volume={105},
  pages={6--10},
  year={2014},
  publisher={Elsevier}
}

@article{robertson1997four,
  title={The four-colour theorem},
  author={Robertson, Neil and Sanders, Daniel and Seymour, Paul and Thomas, Robin},
  journal={Journal of combinatorial theory, Series B},
  volume={70},
  number={1},
  pages={2--44},
  year={1997},
  publisher={Elsevier}
}

@article{appel1977every,
  title={Every planar map is four colorable. {P}art {I}{I}: {R}educibility},
  author={Appel, Kenneth and Haken, Wolfgang and Koch, John},
  journal={Illinois Journal of Mathematics},
  volume={21},
  number={3},
  pages={491--567},
  year={1977},
  publisher={Duke University Press}
}

@article{pach2006many,
  title={How many ways can one draw a graph?},
  author={Pach, J{\'a}nos and T{\'o}th, G{\'e}za},
  journal={Combinatorica},
  volume={26},
  number={5},
  pages={559--576},
  year={2006},
  publisher={Springer}
}

@article{konig1927schlussweise,
  title={{\"U}ber eine Schlussweise aus dem Endlichen ins Unendliche},
  author={K{\"o}nig, D{\'e}nes},
  journal={Acta Sci. Math.(Szeged)},
  volume={3},
  number={2-3},
  pages={121--130},
  year={1927}
}

@article{wagner1967fastplattbare,
  title={Fastpl{\"a}ttbare graphen},
  author={Wagner, Klaus},
  journal={Journal of Combinatorial Theory},
  volume={3},
  number={4},
  pages={326--365},
  year={1967},
  publisher={Elsevier}
}

@article{denise1996random,
  title={The random planar graph},
  author={Denise, Alain and Vasconcellos, Marcio and Welsh, Dominic JA},
  journal={Congressus numerantium},
  pages={61--80},
  year={1996},
  publisher={UTILITAS MATHEMATICA PUBLISHING INC}
}

@article{norine2006proper,
  title={Proper minor-closed families are small},
  author={Norine, Serguei and Seymour, Paul and Thomas, Robin and Wollan, Paul},
  journal={Journal of Combinatorial Theory, Series B},
  volume={96},
  number={5},
  pages={754--757},
  year={2006},
  publisher={Elsevier}
}

@article{manning2005geometry,
  title={Geometry of pseudocharacters},
  author={Manning, Jason Fox},
  journal={Geometry \& Topology},
  volume={9},
  number={2},
  pages={1147--1185},
  year={2005},
  publisher={Mathematical Sciences Publishers}
}

@article{albrechtsen2025fatK2n,
  title={Excluding ${K_{2,t}}$ as a fat minor},
  author={Albrechtsen, Sandra and Distel, Marc and Georgakopoulos, Agelos},
  journal={arXiv preprint arXiv:2510.14644},
  year={2025}
}

@article{saucan2008intrinsic,
  title={Intrinsic differential geometry and the existence of quasimeromorphic mappings},
  author={Saucan, Emil},
  journal={arXiv preprint arXiv:0901.0125},
  year={2008}
}

@article{davies2025riemannian,
  title={Surfaces without quasi-isometric
simplicial triangulations},
  author={Davies, James},
  journal={arXiv preprint arXiv:2603.26652},
  year={2026}
}

@article{ostrovskii2015metric,
  title={Metric dimensions of minor excluded graphs and minor exclusion in groups},
  author={Ostrovskii, Mikhail I and Rosenthal, David},
  journal={International Journal of Algebra and Computation},
  volume={25},
  number={04},
  pages={541--554},
  year={2015},
  publisher={World Scientific}
}

@article{jorgensen2022geodesic,
  title={Geodesic spaces of low {N}agata dimension},
  author={J{\o}rgensen, Martina and Lang, Urs},
  journal={Annales Fennici Mathematici},
  volume={47},
  pages={83--88},
  year={2022}
}

@article{nguyen2025asymptotic,
  title={Asymptotic structure. {I}{I}{I}. {E}xcluding a fat tree},
  author={Nguyen, Tung and Scott, Alex and Seymour, Paul},
  journal={arXiv preprint arXiv:2509.09035},
  year={2025}
}

@article{maillot2001quasi,
  title={Quasi-isometries of groups, graphs and surfaces},
  author={Maillot, Sylvain},
  journal={Commentarii Mathematici Helvetici},
  volume={76},
  number={1},
  pages={29--60},
  year={2001},
  publisher={Springer}
}

@Book{GroAsyInv,
 Author = {Gromov, Mikhael},
 Title = {Geometric group theory. {Volume} 2: {Asymptotic} invariants of infinite groups. {Proceedings} of the symposium held at the {Sussex} {University}, {Brighton}, {July} 14-19, 1991},
 FSeries = {London Mathematical Society Lecture Note Series},
 Series = {Lond. Math. Soc. Lect. Note Ser.},
 ISSN = {0076-0552},
 Volume = {182},
 ISBN = {0-521-44680-5},
 Year = {1993},
 Publisher = {Cambridge: Cambridge University Press},
 Language = {English},
 zbMATH = {437296},
 Zbl = {0841.20039}
}

@article{dunwoody2007planar,
  title={Planar graphs and covers},
  author={Dunwoody, Martin J},
  journal={arXiv preprint arXiv:0708.0920},
  year={2007}
}

@article{hamann2018accessibility,
  title={Accessibility in transitive graphs},
  author={Hamann, Matthias},
  journal={Combinatorica},
  volume={38},
  number={4},
  pages={847--859},
  year={2018},
  publisher={Springer}
}

@article{hamann2018planar,
  title={Planar Transitive Graphs},
  author={Hamann, Matthias},
  journal={The Electronic Journal of Combinatorics},
  volume={25},
  number={4},
  pages={4--8},
  year={2018}
}

@article{esperet2024structure,
  title={The structure of quasi-transitive graphs avoiding a minor with applications to the domino problem},
  author={Esperet, Louis and Giocanti, Ugo and Legrand-Duchesne, Cl{\'e}ment},
  journal={Journal of Combinatorial Theory, Series B},
  volume={169},
  pages={561--613},
  year={2024},
  publisher={Elsevier}
}

@article{mess1988seifert,
  title={The {S}eifert conjecture and groups which are coarse quasiisometric to planes},
  author={Mess, Geoffrey},
  journal={preprint},
  year={1988}
}

@article{tukia1988homeomorphic,
  title={Homeomorphic conjugates of {F}uchsian groups},
  author={Tukia, Pekka},
  journal={Journal f{\"u}r die reine und angewandte Mathematik},
  volume={391},
  pages={1--54},
  year={1988}
}

@article{gabai1992convergence,
  title={Convergence groups are {F}uchsian groups},
  author={Gabai, David},
  journal={Annals of Mathematics},
  volume={136},
  number={3},
  pages={447--510},
  year={1992},
  publisher={JSTOR}
}

@article{casson1994convergence,
  title={Convergence groups and {S}eifert fibered 3-manifolds},
  author={Casson, Andrew and Jungreis, Douglas},
  journal={Inventiones mathematicae},
  volume={118},
  number={1},
  pages={441--456},
  year={1994},
  publisher={Springer}
}

@article{bowditch2004planar,
  title={Planar groups and the {S}eifert conjecture},
  author={Bowditch, Brian H},
  journal={Journal f{\"u}r die reine und angewandte Mathematik},
  volume={576},
  pages={11--62},
  year={2004}
}

@article{maschke1896representation,
  title={The representation of finite groups, especially of the rotation groups of the regular bodies of three-and four-dimensional space, by {C}ayley's color diagrams},
  author={Maschke, Heinrich},
  journal={American Journal of Mathematics},
  volume={18},
  number={2},
  pages={156--194},
  year={1896},
  publisher={JSTOR}
}

@article{georgakopoulos2020planar,
  title={On planar {C}ayley graphs and {K}leinian groups},
  author={Georgakopoulos, Agelos},
  journal={Transactions of the American Mathematical Society},
  volume={373},
  number={7},
  pages={4649--4684},
  year={2020}
}

@article{georgakopoulos2019planar,
  title={The planar {C}ayley graphs are effectively enumerable {I}: consistently planar graphs},
  author={Georgakopoulos, Agelos and Hamann, Matthias},
  journal={Combinatorica},
  volume={39},
  number={5},
  pages={993--1019},
  year={2019},
  publisher={Springer}
}

@article{georgakopoulos2019planar2,
  title={The planar {C}ayley graphs are effectively enumerable {I}{I}},
  author={Georgakopoulos, Agelos and Hamann, Matthias},
  journal={arXiv preprint arXiv:1901.00347},
  year={2019}
}

@article{droms2006infinite,
  title={Infinite-ended groups with planar {C}ayley graphs.},
  author={Droms, Carl and Robinson, DJS},
  journal={Journal of Group Theory},
  volume={9},
  number={4},
  year={2006}
}

@article{arzhantseva2004cayley,
  title={On the {C}ayley graph of a generic finitely presented group},
  author={Arzhantseva, Goulnara N and Cherix, P-A},
  journal={Bulletin of the Belgian Mathematical Society-Simon Stevin},
  volume={11},
  number={4},
  pages={589--601},
  year={2004},
  publisher={The Belgian Mathematical Society}
}

@article{droms1998connectivity,
  title={Connectivity and planarity of {C}ayley graphs},
  author={Droms, Carl and Servatius, Brigitte and Servatius, Herman},
  journal={Beitr{\"a}ge Algebra Geom},
  volume={39},
  number={2},
  pages={269--282},
  year={1998}
}

@article{georgakopoulos2017planar,
  title={The planar cubic {C}ayley graphs},
  author={Georgakopoulos, Agelos},
  journal={Memoirs of the American Mathematical Society},
  volume={250},
  number={1190},
  year={2017},
  publisher={American Mathematical Society}
}

@article{tucker1983finite,
  title={Finite groups acting on surfaces and the genus of a group},
  author={Tucker, Thomas W},
  journal={Journal of Combinatorial Theory, Series B},
  volume={34},
  number={1},
  pages={82--98},
  year={1983},
  publisher={Academic Press}
}

@article{yu1998novikov,
  title={The {N}ovikov conjecture for groups with finite asymptotic dimension},
  author={Yu, Guoliang},
  journal={Annals of Mathematics},
  volume={147},
  number={2},
  pages={325--355},
  year={1998},
  publisher={JSTOR}
}

@article{bernshteyn2025large,
  title={Large-scale geometry of Borel graphs of polynomial growth},
  author={Bernshteyn, Anton and Yu, Jing},
  journal={Advances in Mathematics},
  volume={473},
  pages={110290},
  year={2025},
  publisher={Elsevier}
}

@article{dvovrak2025asymptotic,
  title={Asymptotic dimension of intersection graphs},
  author={Dvo{\v{r}}{\'a}k, Zden{\v{e}}k and Norin, Sergey},
  journal={European Journal of Combinatorics},
  volume={123},
  pages={103631},
  year={2025},
  publisher={Elsevier}
}

@article{gupta2004cuts,
  title={Cuts, trees and $\ell_1$-embeddings of graphs},
  author={Gupta, Anupam and Newman, Ilan and Rabinovich, Yuri and Sinclair, Alistair},
  journal={Combinatorica},
  volume={24},
  number={2},
  pages={233--269},
  year={2004},
  publisher={Springer}
}

@inproceedings{rao1999small,
  title={Small distortion and volume preserving embeddings for planar and Euclidean metrics},
  author={Rao, Satish},
  booktitle={Proceedings of the fifteenth annual symposium on Computational geometry},
  pages={300--306},
  year={1999}
}

@article{bell2008asymptotic,
  title={Asymptotic dimension},
  author={Bell, Greg and Dranishnikov, Alexander},
  journal={Topology and its Applications},
  volume={155},
  number={12},
  pages={1265--1296},
  year={2008},
  publisher={Elsevier}
}

@article{Koebe36,
  author={Koebe, Paul},
  title={Kontaktprobleme der konformen {A}bbildung},
  journal={Berichte \"uber die Verhandlungen der S\"achsischen Akademie der Wissenschaften zu Leipzig, Mathematisch-Physische Klasse},
  volume={88},
  pages={141--164},
  year={1936},
}

@inproceedings{lokshtanov20241,
  title={A 1.9999-approximation algorithm for vertex cover on string graphs},
  author={Lokshtanov, Daniel and Panolan, Fahad and Saurabh, Saket and Xue, Jie and Zehavi, Meirav},
  booktitle={40th International Symposium on Computational Geometry (SoCG 2024)},
  volume={293},
  pages={72},
  year={2024},
  organization={Dagstuhl Publishing}
}

@inproceedings{jia2005universal,
  title={Universal approximations for {T}{S}{P}, {S}teiner tree, and set cover},
  author={Jia, Lujun and Lin, Guolong and Noubir, Guevara and Rajaraman, Rajmohan and Sundaram, Ravi},
  booktitle={Proceedings of the thirty-seventh annual ACM symposium on Theory of computing},
  pages={386--395},
  year={2005}
}

@article{filtser2024scattering,
  title={Scattering and sparse partitions, and their applications},
  author={Filtser, Arnold},
  journal={ACM Transactions on Algorithms},
  volume={20},
  number={4},
  pages={1--42},
  year={2024},
  publisher={ACM New York, NY}
}

@inproceedings{awerbuch1990sparse,
  title={Sparse partitions},
  author={Awerbuch, Baruch and Peleg, David},
  booktitle={Proceedings [1990] 31st Annual Symposium on Foundations of Computer Science},
  pages={503--513},
  year={1990},
  organization={IEEE}
}

@inproceedings{bonamy2025local,
  title={Local Constant Approximation for Dominating Set on Graphs Excluding Large Minors},
  author={Bonamy, Marthe and Gavoille, Cyril and Picavet, Timoth{\'e} and Wesolek, Alexandra},
  booktitle={Proceedings of the ACM Symposium on Principles of Distributed Computing},
  pages={77--87},
  year={2025}
}

@article{bonamy2025distributed,
  title={Distributed Approximation Algorithms for Minimum Dominating Set in Locally Nice Graphs},
  author={Bonamy, Marthe and Gavoille, Cyril and Picavet, Timoth{\'e} and Wesolek, Alexandra},
  journal={arXiv preprint arXiv:2507.04960},
  year={2025}
}

@article{hopf1931ueber,
  title={Ueber den Begriff der vollst{\"a}ndigen differentialgeometrischen Fl{\"a}che},
  author={Hopf, Heinz and Rinow, Willi},
  journal={Commentarii Mathematici Helvetici},
  volume={3},
  number={1},
  pages={209--225},
  year={1931},
  publisher={Springer}
}

@article{menger1927allgemeinen,
  title={Zur allgemeinen kurventheorie},
  author={Menger, Karl},
  journal={Fundamenta mathematicae},
  volume={10},
  number={1},
  pages={96--115},
  year={1927},
  publisher={Polska Akademia Nauk. Instytut Matematyczny PAN}
}

@article{nguyen2025asymptotic4,
  title={Asymptotic structure. IV. A counterexample to the weak coarse Menger conjecture},
  author={Nguyen, Tung and Scott, Alex and Seymour, Paul},
  journal={arXiv preprint arXiv:2508.14332},
  year={2025}
}

@article{nguyen2025asymptotic5,
  title={Asymptotic structure. V. The coarse Menger conjecture in bounded path-width},
  author={Nguyen, Tung and Scott, Alex and Seymour, Paul},
  journal={arXiv preprint arXiv:2509.08762},
  year={2025}
}

@article{nguyen2025asymptotic6,
  title={Asymptotic structure. VI. Distant paths across a disc},
  author={Nguyen, Tung and Scott, Alex and Seymour, Paul},
  journal={arXiv preprint arXiv:2509.07174},
  year={2025}
}

@article{erschler2023assouad,
  title={Assouad--Nagata dimension and gap for ordered metric spaces},
  author={Erschler, Anna and Mitrofanov, Ivan},
  journal={Commentarii Mathematici Helvetici},
  volume={98},
  number={2},
  pages={217--260},
  year={2023}
}

@article{georgakopoulos2026triangulating,
  title={Tirangulating surfaces quasi-isometrically},
  author={Georgakopoulos, Agelos and Vigolo, Federico},
  journal={arXiv preprint arXiv:2603.21189},
  year={2026}
}

@article{ntalampekos2023polyhedral,
  title={Polyhedral approximation of metric surfaces and applications to uniformization},
  author={Ntalampekos, Dimitrios and Romney, Matthew},
  journal={Duke Mathematical Journal},
  volume={172},
  number={9},
  pages={1673--1734},
  year={2023},
  publisher={Duke University Press}
}

@article{creutz2022triangulating,
  title={Triangulating metric surfaces},
  author={Creutz, Paul and Romney, Matthew},
  journal={Proceedings of the London Mathematical Society},
  volume={125},
  number={6},
  pages={1426--1451},
  year={2022},
  publisher={Wiley Online Library}
}

@article{bowditch2020bilipschitz,
  title={Bilipschitz triangulations of riemannian manifolds},
  author={Bowditch, Brian H},
  year={2020}
}

@article{bonk2002quasisymmetric,
  title={Quasisymmetric parametrizations of two-dimensional metric spheres},
  author={Bonk, Mario and Kleiner, Bruce},
  journal={Inventiones mathematicae},
  volume={150},
  number={1},
  pages={127--183},
  year={2002},
  publisher={Springer Science and Business Media LLC}
}

@book{burago2001course,
  title={A course in metric geometry},
  author={Burago, Dmitri and Burago, Yuri and Ivanov, Sergei},
  volume={33},
  year={2001},
  publisher={American Mathematical Soc.}
}

@article{chang20251,
  title={Cutting Planarians: Planar Emulators for String Graphs},
  author={Chang, Hsien-Chih and Conroy, Jonathan and Tan, Zihan and Zheng, Da Wei},
  journal={Accepted to STOC 2026, 58th Annual ACM Symposium on Theory of Computing},
  year={2026}
}

@inproceedings{chang2023covering,
  title={Covering planar metrics (and beyond): {O} (1) trees suffice},
  author={Chang, Hsien-Chih and Conroy, Jonathan and Le, Hung and Milenkovic, Lazar and Solomon, Shay and Than, Cuong},
  booktitle={2023 IEEE 64th Annual Symposium on Foundations of Computer Science (FOCS)},
  pages={2231--2261},
  year={2023},
  organization={IEEE}
}

\end{document}